\documentclass[10pt,a4paper,pointlessnumbers]{scrartcl}

\usepackage[T1]{fontenc}

\usepackage{amssymb}
\usepackage{amsmath}
\usepackage{amsfonts}
\usepackage{bbm}
\usepackage{amsthm}
\usepackage{lmodern}
\usepackage{mathrsfs}
\usepackage[hidelinks]{hyperref}\usepackage{color}
\usepackage[margin=2.5cm]{geometry}
\usepackage[all,cmtip]{xy}
\usepackage[utf8]{inputenc}
\usepackage{graphicx}
\usepackage{varwidth}
\usepackage{overpic}

\makeatletter
\DeclareRobustCommand{\em}{%
	\@nomath\em \if b\expandafter\@car\f@series\@nil
	\normalfont \else \slshape \fi}
\makeatother
\usepackage{upgreek}	
\usepackage{rotating}
\usepackage{tikz}
\usetikzlibrary{matrix,arrows,decorations.pathmorphing,shapes.geometric}
\usepackage{tikz-cd}
\usepackage{needspace}
\newcommand{\spaceplease}{\needspace{5\baselineskip}}
\usetikzlibrary{decorations.markings}

\usetikzlibrary{backgrounds}

\usepackage{todonotes}
\newcommand{\Ao}{\cat{A}_{\text{\normalfont \bfseries !}}}
\newcommand{\Fo}{F_{\text{\normalfont \bfseries !}}}

\newcommand{\vb}{\mathbb{V}}

\newcommand{\ot}{\boxtimes}
\usepackage{tikz-cd}

\usetikzlibrary{decorations.markings}
\usetikzlibrary{backgrounds}

\usepackage{etex}

\tikzstyle{tikzfig}=[baseline=-0.25em,scale=0.5]

\pgfkeys{/tikz/tikzit fill/.initial=0}
\pgfkeys{/tikz/tikzit draw/.initial=0}
\pgfkeys{/tikz/tikzit shape/.initial=0}
\pgfkeys{/tikz/tikzit category/.initial=0}

\pgfdeclarelayer{edgelayer}
\pgfdeclarelayer{nodelayer}
\pgfsetlayers{background,edgelayer,nodelayer,main}

\tikzstyle{none}=[inner sep=0mm]

\newcommand{\tikzfig}[1]{%
	{\tikzstyle{every picture}=[tikzfig]
		\IfFileExists{#1.tikz}
		{\input{#1.tikz}}
		{%
			\IfFileExists{./figures/#1.tikz}
			{\input{./figures/#1.tikz}}
			{\tikz[baseline=-0.5em]{\node[draw=red,font=\color{red},fill=red!10!white] {\textit{#1}};}}%
	}}%
}

\tikzstyle{every loop}=[]

\usepackage{tikzit}
\tikzstyle{black dot}=[fill=black, draw=black, shape=circle, minimum size=3pt, inner sep=0pt]
\tikzstyle{black dot small}=[fill=black, draw=black, shape=circle, minimum size=2pt, inner sep=0pt]
\tikzstyle{fblack dot}=[fill=black, draw=red, shape=circle, minimum size=2pt, inner sep=0pt]
\tikzstyle{wbox}=[fill=white, draw=black, shape=rectangle, minimum height=0.5cm, minimum width=0.01cm]
\tikzstyle{bbox}=[fill=white, draw=blue, shape=rectangle, minimum height=0.5cm, minimum width=0.01cm]
\tikzstyle{rbox}=[fill=white, draw=red, shape=rectangle, minimum height=0.5cm, minimum width=0.01cm]
\tikzstyle{bwbox}=[draw=blue, shape=rectangle, minimum width=2cm, minimum height=0.5cm]
\tikzstyle{bbwbox}=[draw=blue, shape=rectangle, minimum width=1cm, minimum height=1cm]
\tikzstyle{big white circle}=[fill=white, draw=black, shape=circle, minimum width=0.75cm]
\tikzstyle{white dot big}=[fill=white, draw=black, shape=circle, inner sep=1pt]
\tikzstyle{white dot}=[fill=white, draw=black, shape=circle, minimum size=3pt, inner sep=0pt]
\tikzstyle{flat box}=[fill=white, draw=black, shape=rectangle, minimum width=1.3cm, minimum height=0.5cm,fill=morphismcolor]
\tikzstyle{square}=[fill=white, draw=black, shape=rectangle]
\tikzstyle{flat box 2}=[fill=white, draw=black, shape=rectangle, minimum height=0.5cm, minimum width=0.01cm,fill=morphismcolor]
\tikzstyle{bigbox}=[fill=white, draw=black, shape=rectangle, minimum height=0.5cm, minimum width=0.8cm,fill=white]
\tikzstyle{over }=[front]
\tikzstyle{theta}=[fill=blue, draw=blue, shape=ellipse, minimum height=6pt, minimum width=6pt, inner sep=0pt]
\tikzstyle{thetabig}=[fill=blue, draw=blue, shape=ellipse, minimum width=1cm, minimum height=0.01cm]
\tikzstyle{thetainv}=[fill=blue, draw=red, shape=ellipse, minimum height=6pt, minimum width=6pt, inner sep=0pt]
\tikzstyle{thetabinv}=[fill=blue, draw=red, shape=ellipse, minimum width=1cm, minimum height=0.01cm]
\tikzstyle{bigdisk}=[draw=black, shape=circle, minimum width=3cm]
\tikzstyle{wdisk}=[shape=circle, minimum width=0.48cm,fill=white]
\tikzstyle{bigdisk2}=[draw=black, fill=lightgray, shape=circle, minimum width=3cm]
\tikzstyle{little disk}=[fill=white, draw=black, shape=circle, minimum width=0.5cm]

\tikzstyle{mid arrow}=[-, postaction={on each segment={mid arrow}}]
\tikzstyle{end arrow}=[->]
\tikzstyle{mover}=[-, link]
\tikzstyle{string}=[-, draw=blue,postaction={on each segment={mid arrow}}]
\tikzstyle{stringd}=[-, dotted,draw=blue,postaction={on each segment={mid arrow}}]
\tikzstyle{red}=[-, dotted,draw=red]
\tikzstyle{mydots}=[-,dotted,dashed,draw=gray]
\tikzstyle{mydotsblack}=[-,dotted,dashed,draw=black]
\tikzstyle{open}=[-, line width=1pt,draw=blue]
\tikzstyle{thick}=[-,line width=1pt]
\tikzstyle{rarrow}=[->,draw=red]
\tikzstyle{red mid arrow}=[-, draw={rgb,255: red,214; green,42; blue,51}, postaction={on each segment={mid arrow}}, line width=1pt]
\tikzstyle{RED}=[-, draw={rgb,255: red,214; green,42; blue,51}]
\tikzstyle{REDdashed}=[-,dashed, draw={rgb,255: red,214; green,42; blue,51}]
\tikzstyle{REDarrow}=[->, draw={rgb,255: red,214; green,42; blue,51}]
\tikzstyle{darrow}=[->,dotted]
\tikzstyle{blue}=[-, draw=blue]
\tikzstyle{blue mid arrow}=[-, draw={rgb,255: red,23; green,37; blue,167}, postaction={on each segment={mid arrow}}, line width=1pt]
\tikzstyle{over}=[-, link]
\tikzstyle{bover}=[-, blink]
\tikzstyle{mover}=[-, link]
\tikzstyle{mapsto}=[{|->}]

\usetikzlibrary{decorations.pathreplacing,decorations.markings}
\tikzset{
	on each segment/.style={
		decorate,
		decoration={
			show path construction,
			moveto code={},
			lineto code={
				\path [#1]
				(\tikzinputsegmentfirst) -- (\tikzinputsegmentlast);
			},
			curveto code={
				\path [#1] (\tikzinputsegmentfirst)
				.. controls
				(\tikzinputsegmentsupporta) and (\tikzinputsegmentsupportb)
				..
				(\tikzinputsegmentlast);
			},
			closepath code={
				\path [#1]
				(\tikzinputsegmentfirst) -- (\tikzinputsegmentlast);
			},
		},
	},
	mid arrow/.style={postaction={decorate,decoration={
				markings,
				mark=at position .7 with {\arrow[#1]{stealth}}
	}}},
}
\tikzset{%
	link/.style    = { white, double = black, line width = 1.8pt,
		double distance = 0.4pt },
	channel/.style = { white, double = black, line width = 0.8pt,
		double distance = 0.8pt },
}

\tikzstyle{tikzfig}=[baseline=-0.25em,scale=0.5]

\pgfkeys{/tikz/tikzit fill/.initial=0}
\pgfkeys{/tikz/tikzit draw/.initial=0}
\pgfkeys{/tikz/tikzit shape/.initial=0}
\pgfkeys{/tikz/tikzit category/.initial=0}

\pgfdeclarelayer{edgelayer}
\pgfdeclarelayer{nodelayer}
\pgfsetlayers{background,edgelayer,nodelayer,main}

\tikzstyle{none}=[inner sep=0mm]

\tikzstyle{every loop}=[]

\usepackage{mathtools}
\mathtoolsset{showonlyrefs}
\usepackage{epstopdf}

\newtheoremstyle{mytheorem}% name of the style to be used
{\topsep}% measure of space to leave above the theorem. E.g.: 3pt
{\topsep}% measure of space to leave below the theorem. E.g.: 3pt
{\slshape}% name of font to use in the body of the theorem
{0pt}% measure of space to indent
{\bfseries}% name of head font
{.}% punctuation between head and body
{ }% space after theorem head; " " = normal interword space
{\thmname{#1}\thmnumber{ #2}\thmnote{ {\normalfont\slshape(#3)}}}

\newtheoremstyle{mydefinition}% name of the style to be used
{\topsep}% measure of space to leave above the theorem. E.g.: 3pt
{\topsep}% measure of space to leave below the theorem. E.g.: 3pt
{\normalfont}% name of font to use in the body of the theorem
{0pt}% measure of space to indent
{\bfseries}% name of head font
{.}% punctuation between head and body
{ }% space after theorem head; " " = normal interword space
{\thmname{#1}\thmnumber{ #2}\thmnote{ {\normalfont\slshape(#3)}}}

\theoremstyle{mytheorem}
\newtheorem{theorem}{Theorem}[section]

\makeatletter
\newtheorem*{rep@theorem}{\rep@title}
\newcommand{\newreptheorem}[2]{%
	\newenvironment{rep#1}[1]{%
		\def\rep@title{#2 \ref{##1}}%
		\begin{rep@theorem}}%
		{\end{rep@theorem}}}
\makeatother

\newreptheorem{theorem}{Theorem}
\newreptheorem{corollary}{Corollary}
\newtheorem{lemma}[theorem]{Lemma}
\newtheorem{proposition}[theorem]{Proposition}
\newtheorem{corollary}[theorem]{Corollary}
\theoremstyle{mydefinition}
\newtheorem{definition}[theorem]{Definition}
\newtheorem{exx}[theorem]{Example}
\newenvironment{example}
{\pushQED{\qed}\exx}
{\popQED\endexx}
\newtheorem{remm}[theorem]{Remark}
\newenvironment{remark}
{\pushQED{\qed}\remm}
{\popQED\endremm}
\numberwithin{equation}{section}
\usepackage{enumitem}

\newenvironment{pnum}{\begin{enumerate}[topsep=2pt,parsep=2pt,partopsep=2pt,itemsep=0pt,label={(\roman{*})}]}{\end{enumerate}}

\DeclareMathSymbol{\Phiit}{\mathalpha}{letters}{"08}\let\Phi\undefined\newcommand{\Phi}{\Phiit}
\DeclareMathSymbol{\Psiit}{\mathalpha}{letters}{"09}\let\Psi\undefined\newcommand{\Psi}{\Psiit}
\DeclareMathSymbol{\Sigmait}{\mathalpha}{letters}{"06}\let\Sigma\undefined\newcommand{\Sigma}{\Sigmait}
\DeclareMathSymbol{\Xiit}{\mathalpha}{letters}{"04}
\DeclareMathSymbol{\Lambdait}{\mathalpha}{letters}{"03}\let\Lambda\undefined\newcommand{\Lambda}{\Lambdait}
\DeclareMathSymbol{\Piit}{\mathalpha}{letters}{"05}\let\Pi\undefined\newcommand{\Pi}{\Piit}
\DeclareMathSymbol{\Gammait}{\mathalpha}{letters}{"00}\let\Gamma\undefined\newcommand{\Gamma}{\Gammait}
\DeclareMathSymbol{\Omegait}{\mathalpha}{letters}{"0A}\let\Omega\undefined\newcommand{\Omega}{\Omegait}
\DeclareMathSymbol{\Upsilonit}{\mathalpha}{letters}{"07}\let\Upsilon\undefined\newcommand{\Upsilon}{\Upilonit}
\DeclareMathSymbol{\Thetait}{\mathalpha}{letters}{"02}\let\Theta\undefined\newcommand{\Theta}{\Thetait}

\def\Hom{\mathrm{Hom}}
\let\O\undefined\newcommand{\O}{\catf{O}}

\def\End{\catf{End}}
\def\id{\mathrm{id}}

\newcommand{\envo}{\catf{U}_{\!\!\tiny\int} \cat{O}}

\def\dim{\mathrm{dim}}
\let\to\undefined\newcommand{\to}{\longrightarrow}
\let\mapsto\undefined\newcommand{\mapsto}{\longmapsto}
\newcommand{\catf}[1]{\mathsf{#1}}

\newcommand{\Map}{\catf{Map}}
\def\op{\mathrm{op}}
\newcommand{\inthom}{\underline{\catf{Hom}}}

\newcommand{\framed}{\catf{f}E_2}

\newcommand{\ra}[1]{\xrightarrow{\   #1    \ }}

\def\Grpd{\catf{Grpd}}

\newcommand{\Lexf}{\catf{Lex}^\catf{f}}
\newcommand{\Rexf}{\catf{Rex}^\catf{f}}

\def\Cat{\catf{Cat}}
\def\Set{\catf{Set}}

\def\sSet{\catf{sSet}}

\newcommand{\SNC}{\catf{SN}_\cat{C}}

\newcommand{\hocolimsub}[1]{\underset{#1}{\operatorname{hocolim}}\,}

\newcommand{\Lex}{\catf{Lex}}

\newcommand{\omf}{\mathfrak{B}}

\newcommand{\vect}{\catf{vect}}

\newcommand{\cat}[1]{\mathcal{#1}}

\newcommand{\Hbdy}{\catf{Hbdy}}

\newcommand{\Surf}{\catf{Surf}}

\newcommand{\Graphs}{\catf{Graphs}}

\newcommand{\Legs}{\catf{Legs}}

\newcommand{\Calc}{\catf{Calc} }

\newcommand{\RForests}{\catf{RForests}}
\newcommand{\Forests}{\catf{Forests}}

\newcommand{\As}{\catf{As}}

\newcommand{\Envint}{\catf{U}_{\!\text{\footnotesize $\int$}}\,}

\newcommand{\Aut}{\operatorname{Aut}}

\definecolor{Blue}  {rgb} {0.282352,0.239215,0.803921}
\definecolor{Green} {rgb} {0.133333,0.545098,0.133333}
\definecolor{Red}   {rgb} {0.803921,0.000000,0.000000}
\definecolor{Violet}{rgb} {0.580392,0.000000,0.827450}

\setkomafont{caption}{\small\slshape}

\setkomafont{section}{\rmfamily\large}
\setkomafont{subsection}{\rmfamily}
\setkomafont{subsubsection}{\rmfamily}
\setkomafont{subparagraph}{\normalfont\scshape}

\newtheorem*{theorem*}{Theorem}
\newtheorem*{corollary*}{Corollary}

\makeatletter
\renewcommand\section{\@startsection {section}{1}{\z@}%
	{-3.5ex \@plus -1ex \@minus -.2ex}%
	{2.3ex \@plus.2ex}%
	{\normalfont\scshape\centering}}
\makeatother

\usepackage{titlesec}
\titleformat{\subsection}[runin]% runin puts it in the same paragraph
{\normalfont\bfseries}% formatting commands to apply to the whole heading
{\thesubsection}% the label and number
{0.5em}% space between label/number and subsection title
{}% formatting commands applied just to subsection title
[.]% punctuation or other commands following subsection title
\renewcommand\thepart{\Alph{part}}
\titleformat{\part}% runin puts it in the same paragraph
{\Large\bfseries\centering}% formatting commands to apply to the whole heading
{\thepart}% the label and number
{0.5em}% space between label/number and subsection title
{}% formatting commands applied just to subsection title
[]% punctuation or other commands following subsection title

\usepackage{titletoc}
\dottedcontents{part}[0.5cm]{\rmfamily\bfseries}{0.5cm}{0.2cm}
\dottedcontents{section}[1.5cm]{\rmfamily}{0.5cm}{0.2cm}

\setkomafont{disposition}{\normalfont\bfseries}
\begin{document}

	\vspace*{-0.5cm}	\begin{center}	\textbf{\large{The Cyclic and Modular Microcosm Principle in Quantum Topology}}\\	\vspace{1cm}	{\large Lukas Woike }\\ 	\vspace{5mm}{\slshape  Université Bourgogne Europe\\ CNRS\\ IMB UMR 5584\\ F-21000 Dijon\\ France }\end{center}	\vspace{0.3cm}	
	\begin{abstract}\noindent Monoidal categories with additional structure such as a braiding or some form of duality abound in quantum topology. They often appear in tandem with Frobenius algebras inside them. Motivations for this range from the theory of module categories to the construction of correlators in conformal field theory. We generalize the Baez-Dolan microcosm principle to consistently describe all these types of algebras by extending it to cyclic and modular algebras in the sense of Getzler-Kapranov. Our main result links the microcosm principle for cyclic algebras to the one for modular algebras via Costello's modular envelope. The result can be understood as a local-to-global construction or an integration procedure for various flavors of Frobenius algebras that substantially generalizes and unifies the available (and often intrinsically semisimple) methods using for example triangulations or classical skein theory. As the main application of this rather abstract result, we solve the problem of classifying consistent systems of correlators for open conformal field theories and show that the genus zero correlators for logarithmic conformal field theories constructed by Fuchs-Schweigert can be uniquely extended to handlebodies. This establishes a very general correspondence between full genus zero conformal field theory in dimension two and skein theory in dimension three. 	\end{abstract}

	\tableofcontents
	\normalsize

	\section{Introduction and summary}
	It is standard knowledge that a category $\cat{M}$ in which we would like to define the notion of an \emph{associative algebra} $A$ should be \emph{monoidal}. 
	Once we have a monoidal product $\otimes$ on $\cat{M}$ with monoidal unit $I\in\cat{M}$,
	we can define an associative algebra 
	 as an object $A\in \cat{M}$ equipped 
	 with a morphism $\mu : A \otimes A \to A$ called \emph{multiplication} 
	 and a morphism $\eta : I \to A$ called \emph{unit} satisfying associativity and unitality conditions with respect to the monoidal structure of $\cat{M}$; more precisely, the associators and unitors of $\otimes$ tell us \emph{how} to define associativity and unitality, respectively, in this context. 
	A monoidal category is nothing else but an associative algebra 
	in the symmetric monoidal $(2,1)$-category $\Cat$ 
	of categories, where it is crucial that an associative algebra 
	in $\Cat$ is understood \emph{up to coherent homotopy}. 
	As a consequence,
	  an algebra $A$ in $\cat{M}$ can be seen as a nested structure: $A$ is an associative
	algebra inside an associative
	algebra $\cat{M}$, but these two algebras live on different categorical levels. 
	This is by no means an accident, but rather an instance of a profound principle of categorical algebra, the \emph{Baez-Dolan microcosm principle}~\cite[Section~4.3]{baezdolanmicrocosm} that says that an algebraic object $A$ of a certain flavor can be defined in a category $\cat{M}$ that carries the same algebraic structure, but one categorical level higher.
	In other words, $\cat{M}$ is equipped with a categorified version of the same algebraic structure. 
	One calls $A$ the \emph{microcosm} and $\cat{M}$ the \emph{macrocosm}. 
	 An algebraic object `of a certain flavor' 
	  refers in the context of \cite{baezdolanmicrocosm} to an algebraic structure that can be described by an \emph{operad}, an algebraic gadget describing structures featuring operations with multiple inputs and one output.
	  The above-mentioned notion of an algebra in a monoidal category is then obtained by applying the microcosm principle to the associative operad $\As$: An associative
	  algebra $A$ in a monoidal category $\cat{M}$ is an $\As$-algebra inside a homotopy coherent $\As$-algebra in $\Cat$.

	Getzler and Kapranov introduced two important types of operads with extra structure, namely \emph{cyclic operads}~\cite{gk},
	which allow us to cyclically permute the inputs with the output, and \emph{modular operads}~\cite{gkmod}, which come additionally with a self-composition of operations. A dense but self-contained introduction to cyclic and modular operads that needs basically no prior knowledge of ordinary operads is included in Section~\ref{seccycmodop}. 
	In this article, we extend the microcosm principle to cyclic and modular operads and derive, as the main result (Theorem~\ref{exttheorem}), a relation between the cyclic and the modular microcosm principle that has far-reaching applications in quantum algebra and quantum topology.
	
	In more detail, we take as a starting point the notion of a 
	cyclic or modular operad in a symmetric monoidal bicategory $\cat{S}$, where these notions are always understood up to coherent homotopy. 
	This notion comes out of Costello's definition of cyclic and modular operads~\cite{costello} using graph categories;
	in~\cite{cyclic}, this is adapted to the bicategorical case. Relative to this framework, the cyclic and modular microcosm principle is set up~(Section~\ref{seccyclicmodmicrocosmprinciple}):
	Let $\cat{O}$ be a category-valued cyclic or modular operad 
	and $\cat{A}$ a cyclic or modular $\cat{O}$-algebra 
	inside a symmetric monoidal bicategory $\cat{S}$; as usual, everything is defined up to coherent homotopy.
	For specificity, let us think of $\cat{S}$ as a symmetric monoidal bicategory of (finite) linear categories and left exact functors	 with the Deligne product $\boxtimes$ as in~\cite{etingofostrik,egno}, see \cite[Section~2]{fss} for an condensed introduction.
	The data of a cyclic or modular $\cat{O}$-algebra on $\cat{A}$ includes a non-degenerate symmetric pairing $\kappa : \cat{A}\boxtimes\cat{A}\to I$.
	We now consider a \emph{self-dual object} $X$ in $\cat{A}$, i.e.\ an object in $\cat{A}$ equipped with a non-degenerate symmetric pairing with respect to the pairing of $\cat{A}$ (Definition~\ref{defnondeg2pairing}). Then by inserting $X$ into the algebra $\cat{A}$ we obtain a flat vector bundle $\vb_X^\cat{A}$ over a moduli space of the operations of $\cat{O}$ (the detailed statement is Proposition~\ref{propinsertionbundle}). More precisely, for any operation $o$ of total arity $n$, the algebra $\cat{A}$ gives us a left exact functor $\cat{A}_o: \cat{A}^{\boxtimes n}\to \vect$ to 
	the category $\vect$ of finite-dimensional vector spaces over $k$.
	  The fiber of the flat vector bundle $\vb_X^\cat{A}$
	over 
	 $o$ 
	 is the vector space $\cat{A}_o(X,\dots,X)$, where $X$ is inserted into all $n$ slots of $\cat{A}_o$. A cyclic or modular $\cat{O}$-algebra structure on $X$ inside $\cat{A}$ is defined as a parallel section $\xi$
	  of this flat vector bundle;
	 it consists therefore of 
	 vectors
	 \begin{align}
	 	\xi_o  \in \cat{A}_o(X,\dots,X)\quad \text{for all operations $o$.} 
	 	 \label{eqnvectorsofthesection}
	 	\end{align}
	Setting up this framework is tedious, but straightforward. 
	
	We then put the cyclic microcosm principle to the test: Cyclic associative algebras in the symmetric monoidal bicategory $\Lexf$ of finite categories in the sense of \cite{etingofostrik,egno}, left exact functors and linear natural transformations are thanks to \cite[Theorem~4.12]{cyclic} equivalent to pivotal Grothendieck-Verdier categories in $\Lexf$ as defined by Boyarchenko-Drinfeld in~\cite{bd}. These are monoidal categories $\cat{A}$ with a weak, not necessarily rigid type of duality $D:\cat{A}\to\cat{A}^\op$, the \emph{Grothendieck-Verdier duality},
	 whose square comes with a monoidal trivialization to the identity. We give the full definition in Section~\ref{seccycas}, but already mention here that, if the Grothendieck-Verdier duality is rigid and if the category has a simple monoidal unit, one obtains exactly the pivotal finite tensor categories in the sense of~\cite{etingofostrik,egno}.
	 By calculating the cyclic associative algebras inside such a $\Lexf$-valued cyclic associative algebra,
	 we find --- as we should! --- an algebraic structure 
	 that in the rigid case specializes to the notion of a symmetric Frobenius algebra, 
	 see e.g.~\cite{fuchsstigner}, and that, beyond the rigid case, recovers the recently introduced notion of a symmetric Frobenius algebra in a pivotal Grothendieck-Verdier category from~\cite[Section~4.2]{fsswfrobenius}, see Example~\ref{examplesymfrob} for sources of symmetric Frobenius algebras.
	 
	 From a categorical and operadic point of view, it is certainly satisfying to \emph{prove} what the  notion of symmetric Frobenius algebra --- according to the microcosm principle --- needs to be, and it is worth noting  that,
	  even in the rigid case, it requires seeing the rigid duality as a Grothendieck-Verdier duality. Nonetheless, one could argue that one can just guess the notion by making a reasonable ansatz that one then just turns into a definition. Hence, the main question is what one gains from the proof that this notion of symmetric Frobenius algebra inside a pivotal Grothendieck-Verdier category is compliant with the cyclic microcosm principle.
	  Indeed, this article focuses on the algebro-topological ramifications of the interplay between the cyclic and the modular microcosm principle: 
	 Costello~\cite{costello} defines for a $\Cat$-valued cyclic operad $\cat{O}$ a modular operad $\envo$ by freely completing $\cat{O}$ to a modular operad in a homotopy coherent way. The construction is 
	  adapted to the bicategorical situation in~\cite[Section~7.1]{cyclic}. The operad $\envo$ is called the \emph{modular envelope} of $\cat{O}$.
	 By taking the nerve $B$, geometric realization $|-|$ and fundamental groupoid $\Pi$, we obtain a $\Grpd$-valued operad $\Pi |B\envo|$ (the operad obtained by localizing $\envo$ at all morphisms). 
	 A cyclic $\cat{O}$-algebra in a symmetric monoidal bicategory $\cat{S}$
	 extends uniquely to a modular $\Pi |B\envo|$-algebra in $\cat{S}$ 
	 that we call the \emph{modular extension}  $\widehat{\cat{A}}$ of $\cat{A}$, see \cite[Proposition~7.1]{cyclic} and~\cite[Theorem~4.2]{mwansular}.
	 Section~\ref{secmodularextension} gives a concise summary of these previous results.
	Let us now state the central result relating the cyclic and the modular
	version of the microcosm principle:

		\begin{reptheorem}{exttheorem}[Microcosmic version of modular extension]
		For a $\Cat$-valued cyclic	 operad $\cat{O}$ and a
		cyclic  $\cat{O}$-algebra $\cat{A}$ in a symmetric monoidal bicategory $\cat{S}$, there is a canonical pair of inverse equivalences
		\begin{equation}
			\begin{tikzcd}
				\catf{CycAlg}(\cat{O};\cat{A}) \ar[rrrr, shift left=2,"\text{extension}"] &&\simeq&& \ar[llll, shift left=2,"\text{restriction}"] 
				\catf{ModAlg}\left( \Pi|B\envo|; \widehat{\cat{A}}\right) 
			\end{tikzcd}
		\end{equation} 
		between cyclic $\cat{O}$-algebras in $\cat{A}$ and modular $\Pi|B\envo|$-algebras in
		the modular extension $\widehat{\cat{A}}$.
	\end{reptheorem}

The proof of this result is given in Section~\ref{secmodularextension}
once the framework for two layers of cyclic and modular algebras is set up correctly, 
with the suitable notion of non-degenerate symmetric pairing relative to a higher categorical non-degenerate symmetric pairing.
The main contribution of this article is to derive results from several special cases of Theorem~\ref{exttheorem} relevant in quantum topology. These are cases 
 in which all of the quantities appearing in the result can actually be calculated. 
More precisely, Theorem~\ref{exttheorem}
 will be used for concrete applications in the following way:
\begin{itemize}
	\item Choose for $\cat{O}$ a cyclic operad relevant in quantum algebra or topology, such as the associative operad or the framed $E_2$-operad. These operads are `small' in the sense that we have well-known small presentations in terms of generators and relations. The inclusion of the cyclic structure into this presentation is a little more subtle, but doable using~\cite[Section~3]{cyclic}.
	This will then allow us to describe $\catf{CycAlg}(\cat{O};\cat{A})$ explicitly through short lists of generating algebraic operations and their relations.
	
	\item Calculate (or rather: look up) $\Pi|B\envo|$. For the above-mentioned examples of $\cat{O}$, this will produce an interesting operad whose modular algebras are hard to calculate directly. 
	We then use Theorem~\ref{exttheorem} to describe $\catf{ModAlg}\left( \Pi|B\envo|; \widehat{\cat{A}}\right)$ via $\catf{CycAlg}(\cat{O};\cat{A})$.
	\end{itemize}

\subsection*{Correlators for open conformal field theories via `integration' of Frobenius algebras over surfaces}	
Theorem~\ref{exttheorem} can be used to classify correlators for open conformal field theories
 in purely algebraic terms.
 Since this article is mostly written from the perspective of algebra and topology,
 the mathematical physics background will just serve as an extremely important motivation, but no knowledge of conformal field theory (other than the very little background given in the text) will be needed.
 A modern introduction to the notion of a correlator in conformal field theory is~\cite[Section~4.2]{frs25} or the recent  encyclopedia entry~\cite{algcften}. 
 For an open conformal field theory,
 the monodromy data is described by a modular algebra over the 
 \emph{open surface operad} $\O$ 
 (built from compact oriented surfaces with at least one boundary component per connected component and a collection of boundary intervals embedded in their boundary, see Section~\ref{secopen}); in the context of this article, it takes values in $\Lexf$. A modular $\O$-algebra 
 is called a  \emph{categorified open topological field theory} in~\cite{envas}. This is the open version of what is known as a \emph{modular functor}~\cite{Segal,ms89,turaev,tillmann,baki} and describes for us the monodromy data of an open conformal field theory,
 see in addition the  texts~\cite{Lazaroiu,mooresegal,laudapfeiffer} on open(-closed) field theory.
 The reader should be warned that not all definitions of modular functors that abound in the literature agree.
 In this text, an open modular functor is by definition a modular algebra over the open surface operad following~\cite{envas}; it thereby fits into the framework set up for modular functors in~\cite{brochierwoike} and 
  open-closed modular functors in~\cite{sn}.
 In more detail, an open modular functor $\omf$
 assigns to a surface $\Sigma$ with $n$ intervals embedded in its boundary labeled with objects $X_1,\dots,X_n$ in the underlying category $\cat{A}\in\Lexf$ of the modular algebra a vector space  $\omf(\Sigma;X_1,\dots,X_n)$, called \emph{space of conformal blocks}, carrying a representation of the mapping class group $\Map(\Sigma)$
 of $\Sigma$; schematically:
 \begin{equation}
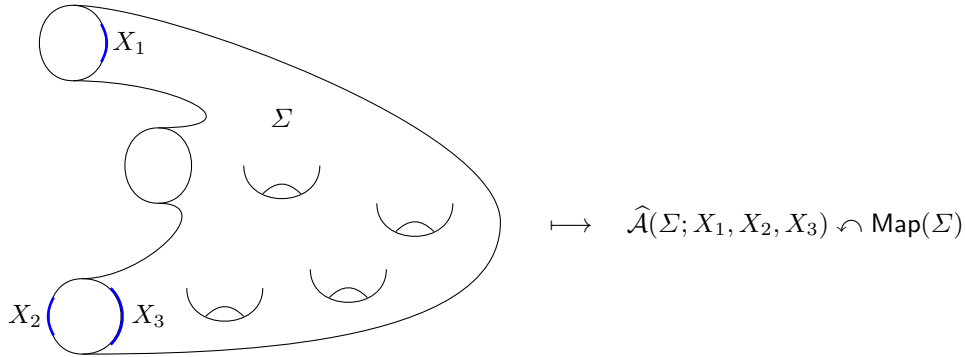
 	\begin{array}{c}			\begin{tikzpicture}[scale=0.5]
 			\begin{pgfonlayer}{nodelayer}
 				\node [style=none] (0) at (-14.25, 11.25) {};
 				\node [style=none] (1) at (-14.25, 9.25) {};
 				\node [style=none] (2) at (-14, 4) {};
 				\node [style=none] (3) at (-14, 2) {};
 				\node [style=none] (4) at (-3, 5.5) {};
 				\node [style=none] (5) at (-11.25, 3.75) {};
 				\node [style=none] (6) at (-9.25, 3.75) {};
 				\node [style=none] (7) at (-10.75, 3) {};
 				\node [style=none] (8) at (-9.75, 3) {};
 				\node [style=none] (9) at (-6.25, 6) {};
 				\node [style=none] (10) at (-4.25, 6) {};
 				\node [style=none] (11) at (-5.75, 5.25) {};
 				\node [style=none] (12) at (-4.75, 5.25) {};
 				\node [style=none] (13) at (-9.75, 7) {};
 				\node [style=none] (14) at (-7.75, 7) {};
 				\node [style=none] (15) at (-9.25, 6.25) {};
 				\node [style=none] (16) at (-8.25, 6.25) {};
 				\node [style=none] (17) at (-12, 8) {};
 				\node [style=none] (18) at (-12, 6) {};
 				\node [style=none] (19) at (-13.5, 9.75) {};
 				\node [style=none] (20) at (-13.5, 10.75) {};
 				\node [style=none] (21) at (-14.75, 2.5) {};
 				\node [style=none] (22) at (-14.75, 3.5) {};
 				\node [style=none] (23) at (-13.25, 3.75) {};
 				\node [style=none] (24) at (-13.25, 2.25) {};
 				\node [style=none] (25) at (-8, 4.25) {};
 				\node [style=none] (26) at (-6, 4.25) {};
 				\node [style=none] (27) at (-7.5, 3.5) {};
 				\node [style=none] (28) at (-6.5, 3.5) {};
 				\node [style=none] (29) at (-12.25, 3) {$X_3$};
 				\node [style=none] (30) at (-15.5, 3) {$X_2$};
 				\node [style=none] (31) at (-12.75, 10.25) {$X_1$};
 				\node [style=none] (32) at (-8.75, 8.25) {$\Sigma$};
 			\end{pgfonlayer}
 			\begin{pgfonlayer}{edgelayer}
 				\draw [bend right=90, looseness=1.50] (1.center) to (0.center);
 				\draw [bend right=270, looseness=1.50] (1.center) to (0.center);
 				\draw [bend left=90, looseness=1.75] (2.center) to (3.center);
 				\draw [bend right=90, looseness=1.50] (2.center) to (3.center);
 				\draw [in=90, out=0, looseness=0.50] (0.center) to (4.center);
 				\draw [in=-90, out=0, looseness=0.75] (3.center) to (4.center);
 				\draw [bend right=90, looseness=1.50] (5.center) to (6.center);
 				\draw [bend left=45, looseness=1.25] (7.center) to (8.center);
 				\draw [bend right=90, looseness=1.50] (9.center) to (10.center);
 				\draw [bend left=45, looseness=1.25] (11.center) to (12.center);
 				\draw [bend right=90, looseness=1.50] (13.center) to (14.center);
 				\draw [bend left=45, looseness=1.25] (15.center) to (16.center);
 				\draw [bend right=90, looseness=1.50] (18.center) to (17.center);
 				\draw [bend right=270, looseness=1.50] (18.center) to (17.center);
 				\draw [style=open, bend right] (19.center) to (20.center);
 				\draw [style=open, bend right] (22.center) to (21.center);
 				\draw [style=open, bend left=45] (23.center) to (24.center);
 				\draw [in=0, out=0, looseness=2.75] (1.center) to (17.center);
 				\draw [in=0, out=0, looseness=1.50] (18.center) to (2.center);
 				\draw [bend right=90, looseness=1.50] (25.center) to (26.center);
 				\draw [bend left=45, looseness=1.25] (27.center) to (28.center);
 			\end{pgfonlayer}
 		\end{tikzpicture}
 	\end{array} \quad \mapsto \quad \omf(\Sigma;X_1,X_2,X_3) \curvearrowleft \Map(\Sigma)
 \end{equation}
 Here the parametrized boundary intervals are printed in blue. Some boundary components may have no parametrized boundary intervals; they are called \emph{free boundary circles}.
 A \emph{consistent system of correlators} for the open modular functor is a modular $\O$-algebra with coefficients in $\omf$. In other words, we use the modular microcosm principle for the definition.
 With the description~\eqref{eqnvectorsofthesection} as a parallel section, this means that such a system of correlators amounts to a self-dual object $F\in \cat{A}$ and vectors
 \begin{align}
 	\xi_\Sigma^F \in \omf(\Sigma;F,\dots,F)    \label{eqncormicrocosm}
 \end{align}	
 in the spaces of conformal blocks $\omf(\Sigma;F,\dots,F)$ for the boundary label $F$ (also called \emph{boundary field})
 that are fixed by the mapping class group action and compatible with the gluing along boundary intervals (one then says that they solve the \emph{sewing constraints}).

 It should be made clear that this does not 
 change or re-invent existing notions of correlators: This definition is in line with the definitions e.g.\ in \cite{ffrsunique,jfcs}, even though we restrict to the open case here and allow more general monodromy data. In the rigid case, the fact that the Frobenius structures describing correlators should be interpreted as a shadow of the duality on $\cat{A}$ under the microcosm principle is briefly commented on in~\cite[Section~4.2]{jfcslog}.

  Equipped with the full description~\eqref{eqncormicrocosm} 
  of correlators through the 
  microcosm principle using the 
  language of cyclic and modular operads 
  at two categorical levels, 
  we can now profit from this new approach.
This will be achieved 
by using the available description of open conformal field theories through
\cite{envas} that builds on
 the results of~\cite{giansiracusa} that in turn 
  rely on the ribbon graph description for moduli spaces of surfaces~\cite{harer86,penner,kontsevichintersection,kon94,costello,costellographs,costellotcft}, see also~\cite{wahlwesterland,egaskupers}.
More precisely,	
we apply Theorem~\ref{exttheorem} to the associative operad and a cyclic associative algebra in $\Lexf$, i.e.\ a pivotal Grothendieck-Verdier category $\cat{A}$. We denote its modular extension  to the modular operad $\Pi|B\Envint\! \As|$ by $\Ao$
(we do not use the notation $\widehat{\cat{A}}$ because we reserve it for the modular extension of ribbon Grothendieck-Verdier categories and want to avoid confusion).
The modular operad 
$\Pi|B\Envint\! \As|$
  is equivalent \cite{giansiracusa,envas}, to the groupoid-valued
	 modular open surface operad $\O$. Therefore, $\Ao$ is an open modular functor, and all open modular functors are of this form~\cite{envas}.
	 
	Theorem~\ref{exttheorem} tells us that any cyclic associative algebra $F$ in $\cat{A}$ (which, as we had seen, is a symmetric Frobenius algebra) extends to a modular algebra $\Fo$ over the open surface operad inside $\Ao$. Therefore, it gives us a consistent system of correlators for the open conformal field theory with monodromy data $\cat{A}$, and  all such systems are of this form. The precise statement is the following:

		\begin{reptheorem}{thmopencorrelators}[Classification of open correlators]
			Given an open modular functor $\omf$ in $\Lexf$, let $\cat{A}$ be the pivotal Grothendieck-Verdier category obtained by evaluation of $\omf$
		on disks with intervals 
		embedded
		in its boundary.
		Then the consistent systems of
		open correlators for $\omf$ are exactly symmetric Frobenius algebras in $\cat{A}$.
		More explicitly,
		the open correlator $\xi^F$
		associated
		to 	a symmetric Frobenius algebra  $F$ in $\cat{A}$
		amounts to
		vectors $\xi_\Sigma^F \in \omf(\Sigma;F,\dots,F)$ in the spaces of conformal blocks associated
		to  surfaces $\Sigma \in \O(n)$ with $F$ appearing $n$ times as the boundary label. 
		These vectors are mapping class group invariant and solve the constraints for sewing along intervals.
	\end{reptheorem}

In the rigid and semisimple case,	the classification
	 of consistent systems of correlators through different flavors of Frobenius algebras in the tensor categories describing the monodromy data is developed in detail in the works~\cite{frs1,frs2,frs3,frs4,ffrs}. 
	  More precisely, the correlators for an open-closed conformal field theory whose monodromy data is described by a modular fusion category are produced via the three-dimensional Reshetikhin-Turaev topological field theory~\cite{rt1,rt2,turaev}. 
	 The fact that one may obtain from a consistent system of open correlators a symmetric Frobenius algebra (one direction of Theorem~\ref{thmopencorrelators})
	 can be deduced 
	 from the results in~\cite{ffrsunique}
	 if $\cat{A}$ is a pivotal fusion category. 
	 Some considerations in the non-semisimple rigid case are made in~\cite{jfcsfincft}.
	 The converse and the extension beyond the semisimple case in full generality seems to be new. 
	 In Section~\ref{seccalc}, we spell out the construction explicitly and discuss briefly the Hopf-algebraic special case.
	 It should also be noted that Theorem~\ref{thmopencorrelators} generalizes, albeit implicitly, the \emph{open} part of the string-net construction of correlators for rational conformal field theories~\cite{rcftsn} beyond the rational case (Remark~\ref{remsn}).

	\subsection*{Ansular correlators}
	In Section~\ref{secansular} we apply the cyclic and the modular microcosm principle to the cyclic framed $E_2$-operad (the cyclic operad of oriented genus zero surfaces). By \cite[Theorem~5.13]{cyclic} cyclic framed $E_2$-algebras in $\Lexf$ are \emph{ribbon Grothendieck-Verdier categories} in $\Lexf$ in the sense of~\cite{bd}, i.e.\ braided monoidal categories $(\cat{A},\otimes,c)$ in $\Lexf$
	with a natural automorphism $\theta_X : X \to X$, called \emph{balancing}, satisfying $\theta_I=\id_I$, $\theta_{X\otimes Y}=c_{Y,X}c_{X,Y}(\theta_X\otimes \theta_Y)$ and $D\theta_X = \theta_{DX}$ for all $X,Y\in \cat{A}$. Cyclic framed $E_2$-algebras in a ribbon Grothendieck-Verdier category are symmetric Frobenius algebras in $\cat{A}$ that are also braided commutative (Proposition~\ref{propribbonfrob}, see Example~\ref{examplesymcomfa} for sources of such algebraic objects). After modular extension, any ribbon Grothendieck-Verdier category $\cat{A}$
	extends to a modular $\Lexf$-valued algebra over the modular operad $\Hbdy$
	 of three-dimensional compact oriented handlebodies~\cite{mwansular}, a so-called \emph{ansular functor} $\widehat{\cat{A}}$
	that provides for us a consistent system of representations $\widehat{\cat{A}}(H;X_1,\dots,X_n)$ of mapping class groups of handlebodies $H$, with labels $X_1,\dots,X_n$ in $\cat{A}$ attached to disks embedded in the boundary surface of $H$. The mapping class groups of handlebodies are the so-called \emph{handlebody groups}, see~\cite{henselprimer} for an introduction. The vector spaces $\widehat{\cat{A}}(H;X_1,\dots,X_n)$ are again referred to as spaces of conformal blocks for $\cat{A}$.
	 A modular handlebody algebra relative to an ansular functor selects invariant vectors in these representations  that are compatible with gluing. In other words, this structure is a handlebody version of the concept of a correlator; for this reason, we refer to it as \emph{ansular correlator}.  With Theorem~\ref{exttheorem}, these can be classified:
	
	\begin{reptheorem}{thmansularcor}[Classification of ansular correlators]
	The consistent systems of
	ansular correlators for an ansular
	functor are exactly symmetric commutative
	Frobenius algebras $F\in\cat{A}$
	in the ribbon Grothendieck-Verdier category $\cat{A}$ obtained from the 
	ansular functor by genus zero restriction.
	More explicitly,
	the ansular correlator $\xi^F$
	associated
	to 	a symmetric commutative Frobenius algebra $F$ in $\cat{A}$
	amounts to
	vectors $\xi_H^F \in \widehat{\cat{A}}(H;F,\dots,F)$ in the spaces of conformal blocks associated
	to handlebodies $H$
	that are handlebody group invariant and solve the constraints for sewing along disks embedded in the boundary surface of the handlebodies.
	\end{reptheorem}

Since ansular correlators are exactly genus zero correlators (Remark~\ref{remgenuszerocorrelators}), one may read
Theorem~\ref{thmansularcor} as a correspondence between full 	genus zero two-dimensional conformal field theory and a three-dimensional version of skein theory. We expand  on this viewpoint in~Remark~\ref{rem2d3d}.

Let us turn to the rigid case and
suppose that the ribbon Grothendieck-Verdier category $\cat{A}$
is actually a finite ribbon category, i.e.\ in particular rigid. Then symmetric commutative Frobenius algebras in $\cat{A}$ are exactly the consistent systems of genus zero correlators from~\cite[Proposition~4.7]{jfcs}. Theorem~\ref{thmansularcor}  tells us that these genus zero correlators extend uniquely to handlebodies and therefore have a substantially larger symmetry group.	
The application to the monoidal unit as symmetric commutative algebra in a ribbon category yields:

	\begin{repcorollary}{corpointing}
		Let $\cat{A}$ be a finite ribbon category.
	Then the spaces of conformal blocks $\widehat{\cat{A}}(H)$ of the ansular functor for $\cat{A}$ evaluated on 
	handlebodies $H$ without embedded disks come
	with a distinguished 
	$\Map(H)$-invariant non-zero
	vector $\xi_H \in \widehat{\cat{A}}(H)$.  
\end{repcorollary}
The vectors $\xi_H$ can be seen as a non-semisimple analogue of the empty skein (Remark~\ref{rememptyskein}).
Phrased geometrically, Corollary~\ref{corpointing} says that the flat vector bundle $H \mapsto \cat{\cat{A}}(H)$ over the moduli space of handlebodies has a trivial line bundle that sits canonically inside.

In the case that the ribbon category is \emph{unimodular}, i.e.\ its distinguished invertible object~\cite{eno-d} $\alpha\in \cat{A}$ controlling the quadruple dual via $-^{\vee\vee\vee\vee}\cong \alpha \otimes - \otimes \alpha^{-1}$ is isomorphic to the monoidal unit, we prove the following result:\label{defunimodular}

	\begin{reptheorem}{thmnotrrep}
	Let $\cat{A}$ be a unimodular finite ribbon category.
	If $\cat{A}\neq \vect$ and $H$ is a handlebody without embedded disks and genus $g\ge 1$,
	then the $\Map(H)$-representation  $\widehat{\cat{A}}(H)$ is not irreducible.
\end{reptheorem}

This tells us that the handlebody group representation on spaces of conformal blocks behave fundamentally differently from the corresponding representations for mapping class group of surfaces: They are, with some exceptions in trivial cases, not irreducible, while for the quantum representations of mapping class groups of surfaces, there are criteria~\cite{andersenfjelstad} for the irreducibility.  Irreducible examples are known, such as those coming from modular fusion categories of Ising type~\cite{JLLSW}.
The proofs of Corollary~\ref{corpointing} and Theorem~\ref{thmnotrrep}
use multiplicative structures on spaces of conformal blocks
built in Section~\ref{secmult} that 
partially extend work of Juhász~\cite[Section~5.2]{juhasz}.

As a closing comment for this introduction,
let us summarize the rationale behind Theorems~\ref{thmopencorrelators} and~\ref{thmansularcor}:
	The study of spaces of conformal blocks belongs to the realm of \emph{chiral} conformal field theory while the construction of consistent systems of correlators allows to pass to \emph{full} conformal field theory, see e.g.\ the overview in the introduction of~\cite{algcften} for further reading.
This means that Theorems~\ref{thmopencorrelators} and~\ref{thmansularcor} are instances of the following correspondence:
\begin{align}\left\{
	\begin{array}{rcl}	\text{macrocosm} & \longleftrightarrow &\text{chiral conformal theory of a certain flavor} \\ &&  \text{(open, closed, genus zero, open-closed, ansular, \dots),}  \\
		\text{microcosm} &  \longleftrightarrow& \text{full conformal field theory of the same flavor.}\end{array}\right\}\label{eqncorrespondences}
\end{align}
As already mentioned, some of the correspondences~\eqref{eqncorrespondences} are implicit in the articles~\cite{frs1,frs2,frs3,frs4,ffrs,ffrsunique,jfcs,jfcslog}, even though they are not conceptualized in the language of cyclic and modular operads.
The key point of this article lies in the treatment of the macrocosmic and the microcosmic level on the \emph{same algebraic footing}, meaning that they are both described by the same cyclic or modular operad, but on different categorical levels.
This inevitably needs Grothendieck-Verdier dualities for the description of the cyclic structure with respect to which the Frobenius algebras need to be considered.

\vspace*{0.2cm}\textsc{Structure of the article.}
The article is split into two parts: \begin{enumerate}
	\item[(A)] Set-up of the operadic framework for the cyclic and modular microcosm principle,
without any particular emphasis on applications, except for the last section in which we discuss the cyclic associative case in detail.
Readers who are interested in categorical foundations for cyclic and modular operads or  Grothendieck-Verdier duality may want to read just part A as a short independent article. 
\item[(B)] Specialization of the general results to operads relevant in quantum topology, with an eye towards applications in conformal field theory. This part depends logically on part~A because  of the necessary formalization of field-theoretic notions through the microcosm principle. Nonetheless, after reading the introduction, one should still be able to start just with part~B. The operadic intricacies are often hidden in the proofs, so one might just take note of the results and then selectively circle back to part~A for technical details.  
\end{enumerate}

		\vspace*{0.2cm}\textsc{Acknowledgments.} 	I would like to thank
		Christoph Schweigert for extremely valuable discussions and explanations regarding correlators, helpful comments on Example~\ref{exmult} and for sharing with me an early draft of~\cite{fsswfrobenius}. The algebraic foundations for Frobenius structures in Grothendieck-Verdier categories developed therein have allowed to write a substantial part of this article in a much more concise way.
		Moreover, I would like to thank Adrien Brochier, Aaron Hofer,
		Lukas Müller, Yang Yang and Deniz Yeral for  helpful discussions and comments related to this project, and Ingo Runkel for helpful remarks with respect to Example~\ref{exmult}.
		Finally, I would like to thank the anonymous referee whose comments have helped improve the manuscript.
	LW gratefully acknowledges support
	by the ANR project CPJ n°ANR-22-CPJ1-0001-01 at the Institut de Mathématiques de Bourgogne (IMB).
	The IMB receives support from the EIPHI Graduate School (contract ANR-17-EURE-0002).
\needspace{10\baselineskip}
\part{The operadic framework}
	\section{Reminder on cyclic and modular operads\label{seccycmodop}}
	Operads were defined by
	 Boardman-Vogt and May~\cite{bv68,mayoperad,bv73}	 to describe algebraic 
	 structures featuring operations with multiple inputs and one output.
	For an operad $\cat{O}$ with values in a (higher) symmetric monoidal category $\cat{S}$ and  $n\ge 0$, we denote by $\cat{O}(n)\in\cat{S}$ 
	the object of $n$-ary operations, i.e.\ operations with $n$ inputs and one output. 
	Through the permutation of the $n$ inputs, the permutation group on $n$ letters acts on $\cat{O}(n)$.
	There is also an operadic composition $\circ_i : \cat{O}(n) \otimes \cat{O}(m) \to \cat{O}(n+m-1)$ inserting an operation of arity $m$ into the $i$-th slot of an $n$-ary operation, where $1\le i\le n$, and an operadic unit in $\cat{O}(1)$ acting neutrally with respect to operadic composition.
	We will not discuss the compatibility of all this data and instead refer e.g.\ to~\cite[Chapters~1-3]{FresseI}. 
	
	An operad that will be of particular relevance in this article is the operad $\framed$ of \emph{framed little disks} whose space $\framed(n)$ of arity $n$ operations is given by the space of embeddings of $n$ two-dimensional disks into a disk that are composed of translations, rescalings and rotations. This operad can be defined for disks of any dimension and is an example of an operad that largely motivated the invention of operads in~\cite{bv68,mayoperad,bv73}.
	Each space $\framed(n)$ is aspherical; more precisely, it is the classifying space of the framed braid group on $n$ strands. This implies that we may see $\framed$ as a groupoid-valued operad. This statement and an	 explicit model is discussed in~\cite{WahlThesis,salvatorewahl}.
	
	The $\framed$-operad is an example of a \emph{cyclic operad} in the sense of 
	Getzler and Kapranov~\cite{gk}. This means that it comes with additional structure that allows us to permute the inputs with the output.
	Seen differently, this means that the distinction between inputs and the output is erased in a consistent way.
	The cyclic operad $\framed$-operad is equivalent to the
	operad of oriented genus zero surfaces.
	
When considering surfaces of all genus, we see that another operadic structure is present, namely a self-composition of operations that glues two boundary components of a connected surface together. This gives us on the surface operad the structure of a \emph{modular operad} in the sense of~\cite{gkmod}.
For the surface operad, we will content ourselves throughout with its groupoid-version because we will restrict  later to bicategorical algebras. The surface operad $\Surf$ has as its groupoid $\Surf(n)$ of arity $n$ operations the groupoid of surfaces (in this article, these are always compact, oriented, smooth) that are connected and have $n+1$ parametrized boundary components, with the morphisms being isotopy classes of diffeomorphisms preserving the orientation and the boundary parametrizations. Phrased differently, morphisms are mapping classes between surfaces, see \cite{farbmargalit} for an introduction to mapping class groups.
	 The gluing of surfaces gives us the operadic composition.

	We will now turn to a precise description
	of cyclic and modular operads, in the context that is needed for this article.
	To this end, we will use the description from~\cite{costello}
	based on categories of graphs.
	A graph is a set $H$ of \emph{half edges}, a set $V$ of \emph{vertices}, a source map $s : H \to V$ sending a half edge to the vertex that it is attached to, and a $\mathbb{Z}_2$-action on $H$ sending a half edge to the half edge that it is glued to. One calls the $\mathbb{Z}_2$-orbits \emph{edges}. An orbit with two elements is called \emph{internal edge}; an orbit with one element is called \emph{(external) leg}. One defines now the category $\Graphs$: Objects are finite disjoint unions of corollas, i.e.\ graphs with one vertex and a finite, possibly empty set of legs glued to it. The morphisms $T\to T'$ are equivalence classes of finite graphs $\Gamma$ (graphs that have finitely many vertices and half edges) with an identification $\alpha_1:T\cong \nu(\Gamma)$ of $T$ with the graph $\nu(\Gamma)$ obtained by cutting open all internal edges of $\Gamma$ and an identification $\alpha_2:T'\cong\pi_0(\Gamma)$ between $T'$ and the graphs constructed by contracting all internal edges of $\Gamma$. Here an identification between graphs is a bijection between the sets of half edges and vertices commuting with the source maps and the $\mathbb{Z}_2$-actions. 
	Two such triples 
	$(\Gamma,\alpha_1,\alpha_2)$ and
	$(\Gamma',\alpha_1',\alpha_2')$ are considered equivalent if there is an identification $T\cong T'$ compatible with $\alpha_1,\alpha_2,\alpha_1',\alpha_2'$.
	We refer to \cite{costello} for the definition of the composition
	(roughly, it is given by the insertion of graphs into each other).
	Disjoint union gives us a symmetric monoidal structure.

  Inside $\Graphs$, there is the symmetric monoidal subcategory $\Forests$ with the same objects, but with only those morphisms coming from graphs that are forests, i.e.\ graphs all of whose connected components are contractible. Now a \emph{cyclic operad} with values in a symmetric monoidal bicategory $\cat{S}$ is defined as a symmetric monoidal functor $\cat{O}:\Forests \to \cat{S}$ while a \emph{modular operad} with values in $\cat{S}$ is defined as a symmetric monoidal functor $\cat{O}:\Graphs\to\cat{S}$.
  Here we understand `symmetric monoidal functor between symmetric monoidal bicategories' ($\Graphs$ is a 1-category, but can be seen as bicategory) always in the weak sense, i.e.\ up to coherent homotopy, see \cite[Section~2]{schommerpries} for the appropriate framework.
  The details of bicategorical operads up to coherent homotopy are developed in~\cite[Section~2.1]{cyclic}.

	The connection to the more traditional description is as follows: Let $\cat{O}$ be an $\cat{S}$-valued cyclic or modular operad and $T$ a corolla with $n\ge 0$ legs. Then $\cat{O}(T)\in\cat{S}$ is the object of operations of arity $n-1$ (which means $n-1$ inputs and one output; we speak in this case of \emph{total arity} $n$). The permutation  action on the $n$ legs of $T$ induces a homotopy coherent 
	action of the permutation group on $n$ letters on $\cat{O}(T)$.
	Since modular operads, as opposed to cyclic ones,
	 are defined on $\Graphs$, we can evaluate on non-contractible morphisms; this gives us exactly the self-composition of operations.
	
	Note that neither for cyclic nor modular operads, the distinction between inputs and outputs is made; in order to make such a choice, one would have to introduce for each corolla a preferred leg called \emph{root}. This gives us the category $\RForests$ of \emph{rooted forests}. A \emph{non-cyclic operad} is a symmetric monoidal functor $\RForests \to \cat{S}$.

	With the definition of cyclic and modular operads given so far, operadic identities (operations of total arity two behaving neutrally with respect to composition) are not yet included, but they can be defined directly \cite[Definition~2.3]{cyclic}, and we assume throughout the article that operads (non-cyclic, cyclic or modular) come with operadic identities.

	\spaceplease
		\section{Flat vector bundles over moduli spaces:	a  description using modular operads}
		For any cyclic or modular operad, we will define in this section
		the notion of a flat vector bundle over the `moduli space of operations of the operad'.
		Basically, this amounts to setting up a categorical or rather operadic version of the Riemann-Hilbert correspondence to bridge between the description of spaces of conformal blocks in~\cite{cyclic,mwansular,brochierwoike}, the description of correlators in~\cite{jfcs} and the microcosm principle~\cite{baezdolan}.
		
	To this end, recall
	 e.g.\ from	\cite[Section I.5]{maclanemoerdijk}
	that for a functor $F:\cat{C}\to\Cat$ the 
	\emph{Grothendieck construction}
	is the category
	$\int F$
	of pairs $(c,x)$ with $c\in \cat{C}$ and $x \in F(c)$. A morphism $(c,x)\to (c',x')$ is a pair $(f,\alpha)$ of a morphism $f:c \to c'$	and a morphism $\alpha :  F(f)x \to x'$. 
	We will also apply the Grothendieck construction
	to \emph{homotopy coherent functors} $F:\cat{C}\to\Cat$ (here $\cat{C}$ is still a 1-category, but seen as a bicategory; $\Cat$ is seen as a bicategory). Note that, in that case, their coherence data will enter into the composition of $\int F$.

	\begin{definition}\label{defopoverop}
		Let $\cat{O}:\Graphs \to \Cat$ be a modular operad. Denote by $\int \cat{O} \in \Cat$ the Grothendieck construction of the symmetric monoidal functor $\cat{O}:\Graphs \to\Cat$ equipped with the symmetric monoidal structure induced by the symmetric monoidal structure of $\Graphs$ and $\cat{O}$.
		We define an \emph{operad over $\cat{O}$} or, for short, \emph{$\cat{O}$-operad with values in a symmetric monoidal category $\cat{K}$}  as a symmetric monoidal functor $\int\cat{O}\to\cat{K}$. 
	\end{definition}
	
	This definition is made analogously for cyclic and ordinary operads instead of modular ones by replacing $\Graphs$ with $\Forests$ and $\RForests$, respectively.

	\begin{definition}[The trivial operad over a modular operad]\label{deftrivialop}
		For a $\Cat$-valued modular operad $\cat{O}$ (or a cyclic or ordinary operad) and any 
		symmetric monoidal category $\cat{K}$, we define the \emph{trivial $\cat{O}$-operad with values in the symmetric monoidal category $\cat{K}$}
		as the symmetric monoidal functor 
		\begin{align}
		\star : \int\cat{O}\to \cat{K}
		\end{align} sending a pair $(T,o)$ of a corolla $T$ and an operation $o\in \cat{O}(T)$
		to the monoidal unit $I\in\cat{K}$ and any morphism to the identity of $I$.
	\end{definition}
	
	Formally, $\star$ depends on $\cat{O}$ and $\cat{K}$, but we suppress this in the notation because it will always be clear from the context. 
	
	In the sequel, we denote by $\vect$ the symmetric monoidal category of finite-dimensional vector spaces over our fixed ground field $k$.

	\begin{definition}[Flat vector bundle over an operad] \label{defflatvectorbundle}
		
		For a $\Cat$-valued modular operad $\cat{O}$, we
		define a \emph{flat vector bundle over $\cat{O}$} as a 
		$\vect$-valued $\cat{O}$-operad, i.e.\ a symmetric monoidal functor $\int \cat{O}\to\vect$.
		If all values are one-dimensional vector spaces, we call the flat vector bundle over $\cat{O}$ a \emph{flat line bundle over $\cat{O}$}.
		The flat line bundle $\int \cat{O}\to\vect$ that is the trivial $\cat{O}$-operad $\star$
		in the sense of Definition~\ref{deftrivialop} will be referred to as the \emph{trivial line bundle over $\cat{O}$} and for brevity just denoted by $k$ (the ground field $k$ is the monoidal unit of $\vect$).
	\end{definition}

	\begin{remark}[Vector bundles over the moduli space of  surfaces or handlebodies]
		Seeing $\vect$-valued functors on groupoids (or even more generally categories), such the ones in Definition~\ref{defflatvectorbundle}, as flat vector bundles is rather a question of perspective and also language; it is the viewpoint taken e.g.\ in~\cite{willerton}. 
		Let us explain why Definition~\ref{defflatvectorbundle} describes really a flat vector bundle in the traditional sense over a moduli space built from the operations of $\cat{O}$, with additional structure taking the gluing operations into account. 
		To illustrate this,
		we choose for $\cat{O}$ the
		 modular operad $\Surf$ of surfaces (see Section~\ref{seccycmodop}).
		Then 
		by	Definition~\ref{defopoverop}
		a flat vector bundle over $\Surf$ is a symmetric monoidal functor $V:\int \Surf \to\vect$. Consider a connected surface $\Sigma_{g,n}$ of genus $g$ with $n\ge 0$ boundary components. Then $\Sigma_{g,n} \in \Surf(T_{n-1})$, where $T_{n-1}$ is the corolla with $n$ legs. 
		The object $(T_{n-1},\Sigma_{g,n})$ is sent by $V$ to a vector space $V(  T_{n-1},\Sigma_{g,n}   )$. 
		A mapping class $f : \Sigma_{g,n}\to\Sigma_{g,n}$ acts by a linear automorphism of $V(  T_{n-1},\Sigma_{g,n}   )$. By functoriality this endows $V(  T_{n-1},\Sigma_{g,n}   )$ with an action of $\Map(\Sigma)$. 
		The mapping class group $\Map(\Sigma)$ is exactly the fundamental group of a suitable moduli space $\cat{M}_{g,n}$ of  surfaces of genus $g$ with $n$ boundary components (or rather moduli stack), see e.g.~the introduction of \cite{egaskupers} for this classical fact, and also~\cite[Chapter~6]{baki}. Therefore, it  follows from the Riemann-Hilbert correspondence 
		that the $\Map(\Sigma)$-action on $V(  T_{n-1},\Sigma_{g,n}   )$ amounts precisely to a flat vector bundle on $\cat{M}_{g,n}$ with typical fiber $V(  T_{n-1},\Sigma_{g,n}   )$ that we denote simply by $V_{g,n}$.  The $\vect$-valued $\Surf$-operad $V$ provides such a flat vector bundle for each $g$ and each $n$, but it actually comes with gluing operations as well: Let $\Gamma : T \to T'$ be an arbitrary morphism in $\Graphs$; we can without loss of generality assume that $T'$ is a corolla (because $V$ is symmetric monoidal). We decompose $T$ into corollas via $T=\sqcup_{j\in J} T^{(j)}$, where $J$ is a finite set. With $n_j := |\Legs(T^{(j)})|$,
		 an operation in $\Surf(T)$ is a family of surfaces $\Sigma_{g_j,n_j}$,
		  with $g_j$ denoting the genus and $n_j$ the number of boundary components. 
		The functor $V$ sends $T$ and these surfaces $\Sigma_{g_j,n_j}$ to the tensor product bundle $\bigotimes_{j\in J} V_{g_j,n_j}$ over $\prod_{j\in J} \cat{M}_{g_j,n_j}$.
		The morphism $\Gamma$ prescribes an  operation applied to the surface $\sqcup_{j\in J}  \Sigma_{g_j,n_j}$ 
		that glues several pairs of boundary components together or acts by permutation on the boundary parametrizations. Let us denote the result of this operation  by $\Sigma_{g,n}$ (we know that this is a connected surface because $T'$ is a corolla; therefore, it will have  some genus $g$ and some number $n$ of boundary components).
		The gluing induces a map $\gamma : \prod_{j\in J} \cat{M}_{g_j,n_j} \to \cat{M}_{g,n}$. 
		The morphism $(T,(\Sigma_{g_j,n_j})_{j\in J}) \to (T', \Sigma_{g,n} )$
		in $\int \Surf$ is sent by $V$ to a map of flat vector bundles 
		$\bigotimes_{j\in J} V_{g_j,n_j}\to \gamma^* V_{g,n}$
		over $\prod_{j\in J} \cat{M}_{g_j,n_j}$, a \emph{gluing map} ($\gamma^*$ denotes the pullback along $\gamma$).
		By the functoriality of $V$ these gluing maps are compatible with the composition of gluings. 
		The same considerations can be made for the modular handlebody operad $\Hbdy$ instead of the surface operad.
		More generally, for any modular operad $\cat{O}$,
		this suggest that we should think of a symmetric monoidal functor $\int \cat{O}\to \vect$ as a flat vector bundle over the moduli space of operations of $\cat{O}$. 	
	\end{remark}

	\section{Self-dual objects}
	In this section, we will use 
	self-dual objects inside modular (and cyclic) algebras to construct flat vector bundles over 
	 modular (and cyclic) operads.
	This definition will need, for any symmetric monoidal bicategory $\cat{S}$, the concept of an $\cat{S}$-valued algebra
	 over a cyclic or modular operad $\cat{O}$ with values in the category $\Cat$ of categories. The detailed definition is given in 
	\cite[Section~2.2-2.4]{cyclic}. We recall here the main points: For a non-cyclic operad $\cat{O}$, an algebra structure on some object is a map from $\cat{O}$ to the endomorphism operad on that object. For cyclic or modular operads, the endomorphism operad needs to be made cyclic or modular, respectively, by means of a non-degenerate symmetric pairing on the underlying object.
	For an object $\cat{A} \in \cat{S}$ in a symmetric monoidal bicategory $\cat{S}$,
		 a \emph{non-degenerate symmetric pairing} is a map $\kappa : \cat{A}\boxtimes\cat{A}\to I$ (here $\boxtimes$ is the monoidal product of $\cat{S}$, and $I$ is the monoidal unit) that exhibits $\cat{A}$ as its own dual in the homotopy category of $\cat{S}$ (this is the non-degeneracy)
	and is equipped with the structure of a homotopy fixed point with respect to the $\mathbb{Z}_2$-action on $\kappa$ via the symmetric braiding (this is the symmetry).
	Thanks to non-degeneracy,
	 there exists an essentially unique map $\Delta : I \to \cat{A}\boxtimes \cat{A}$, the \emph{coevaluation map}, that together with $\kappa$ satisfies the zigzag identities up to isomorphism.
	This allows us to define a symmetric monoidal
	functor $ \End_\kappa^\cat{A}: \Graphs \to \Cat$
	via $\End_\kappa^\cat{A}(T):= \cat{S}( \cat{A}^{\boxtimes\Legs(T)} , I)$ for any corolla $T$, the \emph{endomorphism operad of $(\cat{A},\kappa)$}, a $\Cat$-valued modular operad that can be seen as cyclic operad if needed.
	Here $\cat{A}^{\boxtimes\Legs(T)}$ is the unordered monoidal product of $\cat{A}$ with index set $\Legs(T)$. 
	The definition of $\End_\kappa^\cat{A}$ on morphisms uses the coevaluation.
	We may now define the structure of a cyclic or modular $\cat{O}$-algebra on $\cat{A}\in\cat{S}$ as the choice of a non-degenerate symmetric pairing $\kappa$ on $\cat{A}$ and a map $\cat{O}\to\End_\kappa^\cat{A}$ of $\Cat$-valued cyclic or modular operads. 
	We note that bicategorical cyclic or modular algebras over a cyclic or modular operad form a 2-groupoid~\cite[Proposition~2.18]{cyclic}.

 Our main example for the symmetric monoidal bicategory $\cat{S}$ is the symmetric monoidal bicategory $\Lexf$ of \emph{finite linear categories} (or just \emph{finite categories}) over our algebraically closed field $k$ that we fix throughout:
 \begin{itemize}
 	\item The objects are finite $k$-linear categories~\cite{etingofostrik,egno}, i.e.\ linear abelian categories with finite-dimensional morphism spaces, finitely many isomorphism classes of simple objects, enough projective objects and finite length for every object.
 	\item The 1-morphisms are left exact functors, i.e.\ functors preserving finite limits.
 	\item The 2-morphisms are natural linear transformations.
 	\end{itemize}The monoidal product of $\Lexf$ is the Deligne product with monoidal unit $\vect$, the category of finite-dimensional vector spaces over our fixed algebraically closed field $k$.

 If $\cat{O}$ is a $\Cat$-valued modular operad, a $\Lexf$-valued modular $\cat{O}$-algebra has an underlying finite category $\cat{A} \in \Lexf$ with a non-degenerate symmetric  pairing
  $\kappa : \cat{A}\boxtimes \cat{A} \to \vect$. Since $\kappa$ is left exact, it induces an equivalence $D:\cat{A} \to \cat{A}^\op$ determined by $\kappa(X,Y)\cong\cat{A}(DX,Y)$. The equivalence $D$, also referred to as \emph{duality functor}, comes with an isomorphism $D^2 \cong \id_\cat{A}$ thanks to the symmetry of $\kappa$. 
	The coevaluation 
	$\Delta : \vect \to \cat{A}\boxtimes \cat{A}$ can be identified with an object in $\cat{A}\boxtimes\cat{A}$, namely the coend $\Delta = \int^{X \in \cat{A}} DX \boxtimes X$. For an introduction to coends in finite categories, we refer to~\cite{fss}.
For each operation $o\in \cat{O}(T)$ for a corolla $T$, the map of modular operads from $\cat{O}$ to the endomorphism operad of $(\cat{A},\kappa)$ gives us a left exact functor $\cat{A}_o :\cat{A}^{\boxtimes \Legs(T)}\to\vect$. 
We will not spell out here the compatibility with gluing and instead refer to \cite[Section~2]{cyclic}.

We will now prepare the definition of a self-dual object. First recall that, for a
		 symmetric monoidal bicategory $\cat{S}$ and $\cat{A}\in \cat{S}$,
		a \emph{(generalized) object} in $\cat{A}$ is a 1-morphism $X:I\to \cat{A}$.
		A generalized object in a finite category $\cat{A}\in\Lexf$ is a left exact functor $X:\vect \to \cat{A}$. By additivity this functor is determined by its value $X(k)$ on the ground field $k$.
		Therefore, it amounts to an object in $\cat{A}$ in the usual sense.

	\begin{definition}[Symmetric morphism]
		Let $\cat{S}$ be a symmetric monoidal bicategory.
		For any object $\cat{A} \in \cat{S}$, a symmetric 1-morphism $F: I \to \cat{A}\ot  \cat{A}$, or
		\emph{symmetric map} for short, is a 1-morphism in $\cat{S}$ 
		(it is therefore an object in $\cat{A}\ot  \cat{A}$)
		equipped with the structure of a homotopy fixed point with respect to the $\mathbb{Z}_2$-action on $\cat{S}(I,\cat{A}\ot \cat{A})$ via the symmetric braiding of $\cat{S}$. 	A morphism of symmetric 1-morphisms is a 2-morphism that is a map of homotopy fixed points.
	\end{definition}

	\begin{example}\label{exsymstr1}
		For an object $X:I\to \cat{A}$ in $\cat{A}$, the 1-morphism $X\ot  X : I \simeq I\ot  I \to \cat{A}\ot  \cat{A}$ comes with the structure of a symmetric 1-morphism, and we will always equip it with this structure.
	\end{example}

	\begin{example}\label{exsymstr2}
		It follows from~\cite[Remark~2.11]{cyclic} that
		the coevaluation object $\Delta:I\to\cat{A}\boxtimes\cat{A}$ of a
		non-degenerate symmetric pairing 
		$\kappa : \cat{A}\boxtimes\cat{A}\to I$ on $\cat{A}$
		comes with the structure of a symmetric object in $\cat{A}\boxtimes\cat{A}$.
	\end{example}

	\begin{definition}[Symmetric 2-pairing] Let $\cat{S}$ be a symmetric monoidal bicategory.
		Let  $\cat{A} \in \cat{S}$ be an object with non-degenerate symmetric pairing $\kappa : \cat{A}\ot \cat{A} \to I$.
		A \emph{symmetric 2-pairing} on an object $X:I \to \cat{A}$ in $\cat{A}$ is a symmetric map $\beta : X \ot  X \to \Delta$, where $\Delta: I \to \cat{A}\ot \cat{A}$ is the coevaluation 
		associated to $\kappa$
		(the symmetry property of $\beta$ is with respect to the symmetric structures established on $X\ot  X$ and $\Delta$ in the Examples~\ref{exsymstr1} and~\ref{exsymstr2}, respectively). 
	\end{definition}

	\begin{definition}[Self-dual object]\label{defnondeg2pairing}Let $\cat{A}$ be an object in $\cat{S}$ with a non-degenerate symmetric pairing by $\kappa : \cat{A}\ot \cat{A}\to I$ and coevaluation $\Delta : I\to\cat{A}\ot  \cat{A}$.
	A \emph{self-dual object in $\cat{A}$} is an object $X:I\to\cat{A}$ together with a symmetric 2-pairing $\beta : X\ot  X \to \Delta$ 
	that is \emph{non-degenerate} in the sense 
	that there exists a map $\delta : \id_I \to \kappa \circ (X\ot  X)$, called \emph{symmetric 2-copairing}, 
	such that
	\begin{align}
		X \ra{\delta \ot  X} \kappa (X\ot  X) \ot  X \ra{X \ot  \beta} (\kappa \ot  \id) (X \ot  \Delta) \cong X
	\end{align}
	and
	\begin{align}
		X \ra{X \ot  \delta} X \ot  \kappa (X \ot  X) \ra{\beta \ot  X} (\id\ot  \kappa) (\Delta \ot  X) \cong X
	\end{align}
	are the identity of $X$.
\end{definition}

\begin{remark}\label{rempairingcopairing}
	If a symmetric 2-pairing is non-degenerate, then the  2-copairing, whose existence is required in the above definition, is unique. One could have equivalently dualized the definition: Then the symmetric 2-copairing would  be structure and subject to the non-degeneracy condition that an associated symmetric 2-pairing satisfying the above equations
	exists. This symmetric 2-pairing is then unique.
\end{remark}

	The main result of this section is that self-dual objects 
	produce flat vector bundles:
	
	\begin{proposition}\label{propinsertionbundle}
		Let $\cat{O}$ be a category-valued modular operad and $\cat{A}$ a modular $\cat{O}$-algebra in a symmetric monoidal 
		bicategory $\cat{S}$. Moreover, let $X:I\to \cat{A}$ be a self-dual object and
		$\beta : X  \boxtimes  X \to \Delta$ its  symmetric 2-pairing, where $\Delta$ is the coevaluation object for $\cat{A}$. 
		Then
		\begin{align}
		\vb_X^\cat{A} : \int \cat{O} \to \cat{S}(I,I) \ , \quad (T,o) \mapsto \cat{A}_o\circ X^{\ot  \Legs(T)} \quad \text{for a corolla}\ T \label{eqnrelend}
		\end{align}
		is a symmetric monoidal functor, i.e.\ an $\cat{O}$-operad with values in $\cat{S}(I,I)$. 
	\end{proposition}

	If $\cat{S}=\Lexf$, then $\cat{S}(I,I)=\vect$, so the construction produces 
flat vector bundles over $\cat{O}$ in that case.
The functor $\vb_X^\cat{A}$
depends on $\beta$,
but this is suppressed in the notation because we see the pairing as part of
the self-dual object $X$.
	
	\begin{remark}
		An analogous statement holds for cyclic instead of modular operads.
	\end{remark}

	The proof of Proposition~\ref{propinsertionbundle} will
	make use  of the \emph{calculus construction} from \cite[Section~6]{cyclic} that we briefly recall now; for details, we refer to the original article: 
	For the non-degenerate symmetric pairing $\kappa : \cat{A} \boxtimes \cat{A}\to I$ of $\cat{A}$, there is a category $\cat{E}_\kappa^\cat{A}$ of pairs of $T=\sqcup_{j\in J} T^{(j)} \in \Graphs$ ($J$ is a finite set, and the $T^{(j)}$ are corollas) and families $(Z_j)_{j\in J}$ of 1-morphisms $Z_j : I \to \cat{A}^{\boxtimes \Legs \left( T^{(j)}\right)}$. A morphism
	\begin{align}
	(\Gamma,\alpha):	\left(     S=\sqcup_{\ell\in L} S^{(\ell)},  (Y_\ell)_{\ell \in L}   \right) \to 	\left(     T=\sqcup_{j\in J} T^{(j)},  (Z_j)_{j\in J}   \right)\label{eqngammaalpha}
	\end{align}
	is a morphism $\Gamma : S \to T$ in $\Graphs$
	and for any $j\in J$ a 2-morphism
	$\alpha : \boxtimes_{\ell \in L_j} Y_\ell \to Z_j^\Gamma$ between 1-morphisms
	$I \to \boxtimes_{\ell \in L_j}   \cat{A}^{\boxtimes \Legs    \left (S^{(\ell)}\right)}$,
	where $L_j$ is the preimage of $j$ under the map $L \to J$ induced by $\Gamma$,
	and
	$Z_j^\Gamma$ is defined via
	\begin{align}
	Z_j^\Gamma : I \ra{Z_j} \cat{A}^{\boxtimes \Legs \left(T^{(j)}\right)} 
	\ra{\text{apply $\Delta$ to edges collapsed by $\Gamma$}} \boxtimes_{\ell \in L_j}   \cat{A}^{\boxtimes \Legs    \left(S^{(\ell)}\right)}\ . 
	\end{align}
	The category $\cat{E}_\kappa^\cat{A}$ comes with a projection $\cat{E}_\kappa^\cat{A}\to \Graphs$. 
	The  pullback $\left(\int \cat{O}\right)\times_\Graphs \cat{E}_\kappa^\cat{A}$ (it is also a homotopy pullback because $\int \cat{O}\to \Graphs$ is a Grothendieck fibration) is symmetric monoidal and comes 
	with a symmetric monoidal functor
	\begin{align}
	\Calc_\cat{A}: 
	\left(\int \cat{O}\right)\times_\Graphs \cat{E}_\kappa^\cat{A}
	\to \cat{S}(I,I)\ , \label{eqncalculusfunctor}
	\end{align}
	the so-called \emph{calculus functor}
	sending 
	\begin{align}(T,O,Z)=\left(  \sqcup_{j\in J} T^{(j)}, 
	\left(o_j \in \cat{O}(T^{(j)})\right)_{j\in J}     ,
	(Z_j)_{j\in J} \right)\in \left(\int \cat{O}\right)\times_\Graphs \cat{E}_\kappa^\cat{A}\end{align}
	to
	\begin{align}
	\Calc_\cat{A}(T,O,Z)=	\bigotimes_{j\in J} \cat{A}_{o_j} \circ Z_j \ . \end{align}

	\begin{proof}[\slshape Proof of Proposition~\ref{propinsertionbundle}]
		First let us define a symmetric monoidal functor $K: \int \cat{O} \to \cat{E}_\kappa^\cat{A}$. It sends
		a corolla $T$ and $o\in\cat{O}(T)$ to $(T, X^{\boxtimes  \Legs(T)})$.
		Let now \begin{align}(\Gamma, f): \left(S=\sqcup_{\ell \in L} S^{(\ell)},p \in \cat{O}(S)\right) \to (T, o \in \cat{O}(T))\end{align}
		be a morphism in $\int \cat{O}$ (we can without loss of generality assume that $T$ is a corolla).
		We define $K(\Gamma,f)$ to be the
		morphism $\Gamma :S \to T$ in $\Graphs$
		together with a 2-morphism (to be specified) running from
		$
		\boxtimes_{\ell \in L} X^{\boxtimes \Legs \left(S^{(\ell)}\right)}$ 
		to the same object but with each copy of $X\boxtimes X$ belonging to a pair of edges collapsed by $\Gamma$ being replaced with $\Delta$. 
		Therefore, we can use $\beta : X  \ot  X \to \Delta$ to define $K$ on morphisms. It is a straightforward verification
		that $K$ is in fact a symmetric monoidal functor. 
		
		Now by the universal property of the (homotopy) pullback 
		$\left(\int \cat{O}\right)\times_\Graphs \cat{E}_\kappa^\cat{A}$ the functors
		$K: \int\cat{O}\to\cat{E}_\kappa^\cat{A}$ and $\id:\int \cat{O}\to\int\cat{O}$ 
		induce a functor $\int \cat{O}\to \left(\int \cat{O}\right)\times_\Graphs \cat{E}_\kappa^\cat{A}$ that is also symmetric monoidal.
		We restrict
		the calculus functor~\eqref{eqncalculusfunctor} along this functor and find the desired symmetric monoidal functor~\eqref{eqnrelend}.
	\end{proof}

	\begin{remark}
	The proof of Proposition~\ref{propinsertionbundle} does not need the non-degeneracy of $\beta$. We require it nonetheless because otherwise we lose the possibility to make, if needed, a distinction between inputs and outputs, see the comments after~\cite[Proposition~2.12]{cyclic}.
	\end{remark}
	
\begin{remark}
	If $\cat{O}$ is the surface operad, then $\vb_X^\cat{A}$ is a cyclic analogue of the \emph{pinned block functor}
	from~\cite[Section~3.3]{jfcs}. The crucial technical
	difference is that $\vb_X^\cat{A}$, as opposed to the pinned block functor,
	 contains also the self-duality information of $X$ --- which is of course the main point about $\vb_X^\cat{A}$.
	\end{remark}

	We finish this section by characterizing 
	self-dual objects in a finite category 
	explicitly:

	\begin{lemma}\label{lemmaselfdualinlex}
		Let $\cat{A} \in \Lexf$ be a finite category with non-degenerate symmetric pairing $\kappa$, moreover $D:\cat{A}\to \cat{A}^\op$ the equivalence defined via $\kappa(X,Y)\cong \cat{A}(DX,Y)$ for all $X,Y \in \cat{A}$ which comes with a natural isomorphism $\omega : \id_\cat{A}\to D^2$ induced by the symmetry isomorphisms of $\kappa$.
		Then the structure of a self-dual object on $X\in \cat{A}$ is an isomorphism $\psi : X \to DX$ identifying $X$ with its dual $DX$
		such that
		\begin{align}
		X \ra{\omega_X} 	D^2 X	\ra{D\psi} DX 
		\end{align} 
		coincides with $\psi$.
	\end{lemma}
	
	\begin{proof}
		If we describe the self-duality of $X$ via a symmetric 2-copairing $k\to \kappa(X,X)\cong \cat{A}(DX,X)$, see Remark~\ref{rempairingcopairing}, we observe that it gives us exactly a morphism $DX \to X$. The non-degeneracy entails that this morphism is an isomorphism. The symmetry amounts exactly to the remaining condition mentioned in the Lemma.
	\end{proof}

	\section{The cyclic and modular microcosm principle\label{seccyclicmodmicrocosmprinciple}}
	
The \emph{microcosm principle} developed by Baez-Dolan~\cite[Section~4.3]{baezdolanmicrocosm}
	says, roughly, that in order to be able to define an algebra $A$ over an operad $\cat{O}$ \emph{inside} a category $\cat{A}$, the category $\cat{A}$ needs to be itself an algebra over $\cat{O}$. In this situation, $\cat{A}$ is called the \emph{macrocosm} in which the \emph{microcosm} $A$ lives.
	For instance, in the case $\cat{O}=\As$, this tells us that in order to define what an associative algebra $A$ in a category $\cat{A}$ should be, we need to equip $\cat{A}$ with a monoidal structure. Then we can define what it means for $A \in \cat{A}$ to be an associative algebra relative to the monoidal structure of $\cat{A}$.

	The following precise definition of the microcosm principle builds on the original definition of Baez-Dolan, but includes cyclic and modular operads into the picture using the framework developed for the latter in~\cite{costello,cyclic} and, of course, self-dual objects and their flat vector bundles. Recall that in order to define categorical algebras over cyclic and modular operads, we needed a non-degenerate symmetric pairing to make the endomorphism operad cyclic. In other words, the macrocosm needs a non-degenerate symmetric pairing. It will not come as surprise that also on the microcosmic level we will need a non-degenerate symmetric pairing relative to the macrocosmic pairing.

	\begin{definition}[Modular $\cat{O}$-algebra with coefficients --- Modular Microcosm Principle]\label{defmodalgcoeff}
		Let $\cat{O}$ be a $\Cat$-valued modular operad and $\cat{A}$ an $\cat{S}$-valued modular $\cat{O}$-algebra.
		A \emph{modular $\cat{O}$-algebra (with coefficients) in $\cat{A}$ (or relative to $\cat{A}$)} 
		is \begin{itemize}
			\item a self-dual object $X:I\to \cat{A}$ in $\cat{A}$ (Definition~\ref{defnondeg2pairing}), 
			\item together with a monoidal transformation
			\begin{equation}
				\begin{tikzcd}
					& &  \  \ar[Rightarrow,shorten <= 0.01cm, shorten >= 0.01cm]{dd}{\xi} & & \\  \    
				\int	\mathcal{O}   \ar[bend left=40]{rrrr}{\star} \ar[bend left=-40,swap]{rrrr}{\vb_X^\cat{A} \ \text{from Proposition~\ref{propinsertionbundle}}} & &          & &  \cat{S}(I,I)
					\\ 
					& & \ & &
				\end{tikzcd} 
			\end{equation}
		\end{itemize}
		satisfying the \emph{unitality condition} that the component \begin{align}
			\xi_{(T_1,1_\cat{O})} : \star \to \vb_X^\cat{A} (1_\cat{O}) = \cat{A}_{1_\cat{O}} \circ (X\ot  X) \cong \kappa \circ (X\ot  X)\end{align} selects the 2-copairing of the self-dual object;
		here $1_\cat{O}$ is the operadic identity, and the isomorphism $\cat{A}_{1_\cat{O}}\cong \kappa$ is part of the data of $\cat{A}$~\cite[Definition~2.13]{cyclic}.
	\end{definition}

\begin{remark}
	In this formulation of the microcosm principle, the modular algebra with coefficients in a higher categorical modular algebra is defined as a monoidal transformation from $\star$ (for appropriate choices of the underlying operad to be understood as the `trivial field theory') to a symmetric monoidal functor built from the coefficients. 
	Similar concepts appear sometimes under the name of \emph{twisted} or \emph{relative field theory}~\cite{stolzteichner,freedteleman,johnsonfreydscheimbauer}. It is also an abstraction of the definition of correlators in~\cite{jfcs}, as we explain in~Section~\ref{secopen}.
	\end{remark}
	
	\begin{remark}\label{remgrpd}
		Modular algebras with fixed coefficients naturally form a category:
		A morphism $(X,\xi) \to (Y,\xi')$ between modular $\cat{O}$-algebras with coefficients in $\cat{A}$ is a 2-morphism $f:X \to Y$ between the 1-morphisms $X,Y : I \to \cat{A}$ in $\cat{S}$ such that:
		\begin{itemize}
			
			\item The 2-pairing $\beta :X \ot  X \to \Delta$ of $X$ and the 2-pairing $\gamma : Y \ot  Y \to \Delta$ of $Y$ satisfy $\beta = \gamma \circ (f\ot  f)$.
			
			\item The monoidal transformation $f_* : \vb_X^\cat{A}\to\vb_Y^\cat{A}$ induced by $f$ satisfies $f_*\circ  \xi = \xi'$.
		\end{itemize}
		We denote the category of modular $\cat{O}$-algebras in $\cat{A}$ by $\catf{ModAlg}(\cat{O};\cat{A})$.
		As one would expect,
		 any morphism between self-dual objects that is compatible with the respective non-degenerate 2-pairings is invertible. This is essentially the argument from \cite[Proposition~2.18]{cyclic}, but one categorical level lower.
		Hence, the category
		$\catf{ModAlg}(\cat{O};\cat{A})$ is a groupoid.
	\end{remark}
	
	\begin{remark}\label{remgenopcomp}
		For any $o\in \cat{O}(T)$, any morphism $\Gamma : T \to T'$ in $\Graphs$
		induces a morphism $(T,o)\to (T',\cat{O}(\Gamma)o)$ in $\int \cat{O}$ that we just denote by $\Gamma$ again.
		For an $\cat{O}$-algebra $\xi:\star \to \vb_X^\cat{A}$ with coefficients in $\cat{A}$, the naturality of $\xi$ implies that the diagram
		\begin{equation}
			\begin{tikzcd}
				\star   \ar[]{rr}{\xi_{(T,o)}} \ar[swap]{rrdd}{\xi_{(T',\cat{O}(\Gamma)o)}} & &   \cat{A}_o \circ X^{\ot  \Legs(T)}  \ar{dd}{     \vb_X^\cat{A}(\Gamma)  }    \\
				\\
				& &  
				\cat{A}_{(T',\cat{O}(\Gamma)o)} \circ X^{\ot  \Legs(T')}  
			\end{tikzcd} 
		\end{equation}
		on $\cat{S}(I,I)$ commutes.
		This implies that $\xi$ is already determined by those components $\xi_{(T,o)}$ on operations $o\in\cat{O}(T)$ that generate the operations of $\cat{O}$ under operadic composition. In fact, we can further reduce this to the case where $T$ is not an arbitrary object in $\Graphs$, but a corolla. This is because of the monoidality of $\xi$. 
	\end{remark}
	
	\begin{remark}[Cyclic microcosm principle]\label{remcyclicalg}
		Definition~\ref{defmodalgcoeff} has a cyclic analogue. We then get the notion 
		of a cyclic algebra with coefficients in a cyclic algebra over a cyclic operad via the \emph{cyclic microcosm principle}. 
		Remark~\ref{remgrpd} and~\ref{remgenopcomp} remain correct after some straightforward changes.
	\end{remark}	
	
	\begin{remark}
		We should mention that there is a connection between cyclic operads and the microcosm principle that is entirely different from the objective of this text and covered in~\cite{Obradovic}. Let us explain this: Operads themselves can be defined as monoids with respect to Day convolution \emph{inside} a certain symmetric monoidal category, so they can be defined through the microcosm principle. In~\cite{Obradovic} this description is extended to cyclic operads. This is \emph{not} our objective: We fix a cyclic or modular operad $\cat{O}$, while being agnostic to \emph{how} the notion of cyclic or modular operad is defined (whether this is again an instance of the microcosm principle is irrelevant to us), and a cyclic or modular algebra $\cat{A}$ over $\cat{O}$. We then consider cyclic or modular $\cat{O}$-algebras \emph{inside} $\cat{A}$. 
	\end{remark}

	\section{Microcosmic modular extension\label{secmodularextension}}
	In the subsequent sections,
	our goal will be to apply the modular microcosm principle in relevant cases. This means that, for a given modular operad $\cat{O}$  of interest and a modular algebra $\cat{A}$ over it,
	we would like to classify all modular $\cat{O}$-algebras in $\cat{A}$. 
	This is a rather non-trivial task for which we need to establish an important technical result in this section.
	
	First recall that for a $\Cat$-valued cyclic operad $\cat{O}:\Forests \to \Cat$, one can define following \cite{costello} the \emph{modular envelope} of $\cat{O}$, a modular operad $\envo:\Graphs \to \Cat$ sending $T\in \Graphs$ to the Grothendieck construction
	\begin{align}
		\envo(T):= \int \left( \Forests / T \to \Forests \ra{\cat{O}} \Cat   \right) \ , 
		\end{align} 
	where $\Forests / T$ is the slice category of the inclusion $\Forests \to \Graphs$ over $T$, and $\Forests / T \to \Forests$ is the forgetful functor.
	Intuitively, $\envo$ is the modular operad obtained by completing the cyclic operad $\cat{O}$, in a homotopically correct way, to a modular operad.
	The version of the modular envelope used here, while still largely following \cite{costello}, is slightly adapted to $\Cat$-valued operads and their bicategorical algebras, see \cite[Section~7.1]{cyclic}. 
	In particular, the construction is made such that a cyclic $\cat{O}$-algebra gives rise to a modular $\envo$-algebra $\widehat{\cat{A}}$ that is called the \emph{modular extension} of $\cat{A}$. 
On an operation $\left(\Gamma : \widetilde T = \sqcup_{\ell \in L} \widetilde T^{(\ell)} \to T , o \in \cat{O}\left(\widetilde T\right)\right)$, it is given by image of $\cat{A}_o \in \prod_{\ell\in L} \cat{S}\left(\cat{A}^{  \boxtimes \Legs \left( \widetilde T^{(\ell)}   \right)    },I\right)$ under
the functor
\begin{align}
	\prod_{\ell\in L} \cat{S}\left(\cat{A}^{  \boxtimes \Legs \left( \widetilde T^{(\ell)}   \right)    },I\right) \ra{\End_\kappa^\cat{A}(\Gamma)}
	\cat{S} \left(\cat{A}^{\boxtimes \Legs(T)} ,I\right) \end{align}
	obtained by the evaluation of the endomorphism operad for $\cat{A}$ on $\Gamma$.

	The following is the main technical result of this paper:

	\begin{theorem}[Microcosmic version of modular extension]\label{exttheorem}
		For a $\Cat$-valued cyclic	 operad $\cat{O}$ and a
		cyclic  $\cat{O}$-algebra $\cat{A}$ in a symmetric monoidal bicategory $\cat{S}$, there is a canonical pair of inverse equivalences
		\begin{equation}
			\begin{tikzcd}
				\catf{CycAlg}(\cat{O};\cat{A}) \ar[rrrr, shift left=2,"\text{extension}"] &&\simeq&& \ar[llll, shift left=2,"\text{restriction}"] 
				\catf{ModAlg}\left( \Pi|B\envo|; \widehat{\cat{A}}\right) 
			\end{tikzcd}
		\end{equation} 
		between cyclic $\cat{O}$-algebras in $\cat{A}$ and modular $\Pi|B\envo|$-algebras in
		the modular extension $\widehat{\cat{A}}$.
	\end{theorem}

Here $B:\Cat \to \sSet$ is the nerve, $|-|:\sSet \to \catf{Top}$ the geometric realization and $\Pi : \catf{Top}\to\Grpd$ the fundamental groupoid.
	
	\begin{remark}\label{remwithoutloc}
		The modular extension $\widehat{\cat{A}}$ of $\cat{A}$ is a modular $\envo$-algebra sending all morphisms in the categories of operations to isomorphisms, see \cite[Remark~7.2]{cyclic}. Therefore, it can be seen as a modular algebra over $\Pi|B\envo|$ (the localization of $\envo$ at all morphisms in the categories of operations). 
		This however does not influence the notion
		 of a modular algebra with coefficients in $\widehat{\cat{A}}$.
		In other words, $\catf{ModAlg}\left( \Pi|B\envo|; \widehat{\cat{A}}\right)=\catf{ModAlg}\left( \envo; \widehat{\cat{A}}\right)$, so that we can rephrase Theorem~\ref{exttheorem} equivalently as an equivalence
		\begin{equation}\label{eqnequivwithoutloc}
			\begin{tikzcd}
				\catf{CycAlg}(\cat{O} ;\cat{A}) \ar[rrrr, shift left=2,"\text{extension}"] &&\simeq&& \ar[llll, shift left=2,"\text{restriction}"] 
				\catf{ModAlg}\left( \envo; \widehat{\cat{A}}\right) \ . 
			\end{tikzcd}
	\end{equation}\end{remark}

	\begin{proof}[\slshape Proof of Theorem~\ref{exttheorem}]
		By Remark~\ref{remwithoutloc} we can equivalently prove the equivalence~\eqref{eqnequivwithoutloc}.

		We start by explaining why there is a restriction functor $\catf{ModAlg}( \envo; \widehat{\cat{A}}) \to \catf{CycAlg}(\cat{O};\cat{A})$:
		First we observe that being a self-dual object $X: I \to \cat{A}$ in $\cat{A}$ is the same as being a self-dual object $X:I\to \widehat{\cat{A}}$ in $\widehat{\cat{A}}$
		because $\cat{A}$ and its modular extension $\widehat{\cat{A}}$ have the same underlying object in $\cat{S}$ and the same non-degenerate symmetric pairing. 
		Moreover, there is a symmetric monoidal functor $J: \int \cat{O} \to \int \envo$ sending $(T,o)$ with $o\in \cat{O}(T)$ to $(T,\id_T,o) \in \int \envo$. Clearly, for the flat vector bundles from
		Proposition~\ref{propinsertionbundle} \begin{align}\label{eqnendJ}\vb_X^{\widehat{\cat{A}}} \circ J=\vb_X^\cat{A}  \ . \end{align} 
		The restriction functor \begin{align}
	\label{eqnJrestriction}	J^*:	\catf{ModAlg}\left( \envo; \widehat{\cat{A}}\right) \to
			\catf{CycAlg}(\cat{O} ;\cat{A}) \end{align}
			takes an object $\xi \in \catf{ModAlg}( \envo; \widehat{\cat{A}})$, i.e.\ a monoidal transformation $\xi : \star \to \vb_X^{\widehat{\cat{A}}}$ for some self-dual object $X$ in $\widehat{\cat{A}}$, 
		and restricts along $J$. Since $\star \circ J=\star$ and \eqref{eqnendJ}, this produces a monoidal transformation $\star \to \vb_X^\cat{A}$. After the straightforward verification that this restriction preserves the unitality condition from Definition~\ref{defmodalgcoeff}, we obtain an object in $\catf{CycAlg}(\cat{O};\cat{A})$. This establishes the restriction functor~\eqref{eqnJrestriction}.
		
		Now suppose that we are given a cyclic algebra $\xi : \star \to \vb_X^\cat{A}$ with coefficients in $\cat{A}$.
		An operation in $\envo(T')$, where $T'$ is without loss of generality a corolla, is a morphism $\Gamma : T \to T'$ for some $T\in \Graphs$ and an operation
		$o\in \cat{O}(T)$. 
		We decompose $T$ into corollas, $T = \sqcup_{\ell \in L} T^{(\ell)}$, so that $o$ corresponds to a family $(o_\ell)_{\ell \in L} \in \prod_{\ell \in L} \cat{O}\left(T^{(\ell)}\right) \simeq \cat{O}(T)$ of operations. 
		The graph $\Gamma$ provides a morphism $\Gamma : (T,\id_T,o)\to (T',\Gamma,o)$ in $\int \envo$, again denoted by $\Gamma$ by slight abuse of notation.
		We now define $\widehat \xi : \star \to \vb_X^{\widehat{\cat{A}}}$ as the transformation whose component at $(\Gamma,o)$ is given by the commutativity of the diagram
		\begin{equation}\label{trianglemicroeqn}
			\begin{tikzcd}
				\star   \ar[]{rrrr}{\bigotimes_{\ell \in L} \xi_{o_\ell}} \ar[swap]{rrrrdd}{\widehat \xi_{(\Gamma,o)} } & &&& \bigotimes_{\ell \in L} \cat{A}_{o_\ell} 
				\circ X^{\boxtimes     \Legs T^{(\ell)}    }   =  \bigotimes_{\ell\in L} \widehat{\cat{A}}_{(\id_{T^{(\ell)}},o_\ell)}
				\circ X^{\boxtimes     \Legs T^{(\ell)}    } 
				\ar{dd}{     \vb_X^{\widehat{\cat{A}}}(\Gamma)  }    \\
				\\
				& &  &&
				\widehat{\cat{A}}_{(\Gamma,o)} \circ X^{\ot  \Legs(T)} \ . 
			\end{tikzcd} 
		\end{equation}
		This makes $\widehat \xi$ natural, and we extend it monoidally. Moreover, $\widehat \xi$ clearly extends $\xi$, i.e.\ $J^* \widehat \xi = \xi$. In fact, any extension of $\xi$ to a monoidal natural transformation
		$\widehat \xi : \star \to \vb_X^{\widehat{\cat{A}}}$ must make the  triangle~\eqref{trianglemicroeqn} commute as follows from Remark~\ref{remgenopcomp}; in other words, $\widehat \xi$ is the unique extension.
		If we now define an extension functor $E: \catf{CycAlg}(\cat{O};\cat{A}) \to \catf{ModAlg}( \envo; \widehat{\cat{A}}) $ by $E \xi := \widehat \xi$ (the definition clearly extends to morphisms),
		then 
		$J^* \circ E \cong \id_{ \catf{CycAlg}(\cat{O};\cat{A}) }$ because $E$ provides, as just discussed, an extension, and $E \circ J^* \cong \id_{   \catf{ModAlg}( \envo; \widehat{\cat{A}}) }$ because this extension is unique.
	\end{proof}

	\section{Cyclic associative algebras\label{seccycas}}
	In this section, we apply the cyclic microcosm principle to the associative operad, 
	the cyclic $\Set$-valued operad $\As$ sending a corolla $T$ to the set of cyclic orders on $\Legs(T)$.
	If $T$ has $n+1$ legs, $\As(T)$ can be non-canonically identified with the permutation group on $n$ letters.
	Of course, we can see $\As$ as $\Cat$-valued cyclic operad.
	The reader will not see any surprises in this section because cyclic associative algebras will have to amount to symmetric Frobenius algebras, with the right definition of symmetric Frobenius algebra that we will give below in Definition~\ref{defsymFA}. In that regard, this section is rather a sanity check for the cyclic microcosm principle set up above.

	By~\cite[Theorem~4.12]{cyclic} cyclic $\As$-algebras in $\Lexf$ are pivotal Grothendieck-Verdier categories in $\Lexf$, as defined by Boyarchenko-Drinfeld in~\cite{bd}, using the notion of $\star$-autonomous categories~\cite{barr}. A \emph{Grothendieck-Verdier category} in $\Lexf$ (note that our conventions are dual to the ones in~\cite{bd})
	is a monoidal category $\cat{A}$ in $\Lexf$  with monoidal product $\otimes:\cat{A}\boxtimes\cat{A}\to\cat{A}$
	(this is exactly a non-cyclic associative algebra structure)
	 together with an object $K\in\cat{A}$, the \emph{dualizing object},
	  such that the functors $\cat{A}(K,X\otimes -)$ become representable via $\cat{A}(K,X\otimes -)\cong \cat{A}(DX,-)$ with $D:\cat{A}\to\cat{A}^\op$ being an equivalence. 
	The functor $D$ is called \emph{Grothendieck-Verdier duality}; it sends the monoidal unit $I\in \cat{A}$ to $K$. Moreover, there are canonical isomorphisms $D^2I\cong I$ and $D^2 K\cong K$. A \emph{pivotal structure} on a Grothendieck-Verdier category is a monoidal isomorphism $\omega:\id_\cat{A}\cong D^2$ whose component at $K$ is the canonical isomorphism $D^2 K\cong K$. 
	Under the equivalence from pivotal Grothendieck-Verdier categories in $\Lexf$ with cyclic $\As$-algebras in $\Lexf$, a pivotal Grothendieck-Verdier category $\cat{A}$ needs to be equipped 
	 with a non-degenerate pairing. This pairing is $\kappa(-,-)=\cat{A}(D-,-)$ with its symmetry isomorphism coming from the pivotal structure. 
	 A pivotal structure can be equivalently described through isomorphisms $\psi_{X,Y}:\cat{A}(K,X \otimes Y)\cong \cat{A}(K,Y\otimes X)$ squaring to the identity and subject to a cocycle condition~\cite[Section~6]{bd}.
	
		We should point out that the Grothendieck-Verdier duality can be, but need not be a \emph{rigid duality}. In the case of a rigid duality, $D$ sends $X \in \cat{A}$ to an object $X^\vee$ for which we can find a morphism $d:X^\vee\otimes X \to I$ as well as a morphism $b: I\to X\otimes X^\vee$ such that the zigzag identities are satisfied. If the Grothendieck-Verdier duality of a pivotal Grothendieck-Verdier category in $\Lexf$ is rigid and if its unit is simple, it is a \emph{pivotal finite tensor category} in the sense of~\cite{etingofostrik,egno}.
		Note that in principle, one can define two different versions of a rigid duality, a left and a right one. For pivotal finite tensor categories, these canonically agree, so we refrain from making this distinction.
		It should be noted that, even in Grothendieck-Verdier categories, generalizations of the evaluation and coevaluation exist, but these satisfy substantially weaker properties, see \cite[Theorem 4.5]{cs97} and \cite[Definition 3.33]{fsswfrobenius}.

	By~\cite[Section~3.1]{bd} a Grothendieck-Verdier category $\cat{A}$ has a
	\emph{second monoidal product} defined via
	\begin{align} X \odot Y := D^{-1}(DY \otimes DX) \label{eqn2ndmonoidalproduct}\end{align} for all $X,Y \in \cat{A}$. 
	This is not necessarily
	a monoidal product in $\Lexf$,
	but it is right exact. If $D$ is rigid, then $\odot \cong \otimes$. 
	
	With this second monoidal product, we can characterize self-dual objects in a cyclic associative algebra in $\Lexf$. We find a structure that is dual to the notion of a Grothendieck-Verdier pairing in~\cite[Definition~4.7]{fsswfrobenius}.
	
	\begin{lemma}[Characterization of self-dual objects with coefficients in a cyclic associative algebra / pivotal Grothendieck-Verdier category]\label{lemmaequivdes}
		Let $\cat{A}$ be a 
		pivotal Grothendieck-Verdier category in $\Lexf$.
		The structure of a self-dual object on some $X\in \cat{A}$ amounts exactly to morphisms $\beta :    X \odot X \to I$ and $\delta : K \to X \otimes X$ such that $\delta$ 	
		is a fixed point of the $\mathbb{Z}_2$-action on $\cat{A}(K,X\otimes X)$ via the pivotal structure
		and such that
		\begin{align}
		\cat{A}(K,X \otimes X)\otimes \cat{A}(X\odot X,I)\cong \cat{A}(DX,X)\otimes \cat{A}(X,DX) \ra{\circ_X} \cat{A}(DX,DX)
		\end{align}
		sends $\delta \otimes \beta$ to $\id_{DX}$
		while
		\begin{align}
		\cat{A}(X\odot X,I)	\otimes  	\cat{A}(K,X \otimes X)\cong
		\cat{A}(X,DX)	\otimes
		\cat{A}(DX,X) \ra{\circ_{DX}} \cat{A}(X,X)
		\end{align}
		sends $\beta \otimes \delta$ to $\id_X$.
		\label{lemmaequivdesii}
	\end{lemma}

	Justified by this result, we will refer to a map $\beta : X \odot X \to I$ as in Lemma~\ref{lemmaequivdesii} as a \emph{non-degenerate symmetric pairing on $X$}; we call
	$\delta : K	 \to X \otimes X$ 
	 the associated \emph{non-degenerate symmetric copairing}.  
	
	\begin{proof}
		We have $\kappa(X,X)\cong \cat{A}(K,X \otimes X)$, and the symmetry isomorphism of $\kappa$ amounts to the $\mathbb{Z}_2$-action on $\cat{A}(K,X \otimes X)$ by the  pivotal structure of $D$.
		Similarly, \begin{align} (\cat{A}\boxtimes \cat{A})(X \boxtimes X,\Delta)&=\cat{A}(X,\Delta')\otimes\cat{A}(X,\Delta'') \qquad \text{with Sweedler notation $\Delta = \Delta' \boxtimes\Delta''$} \\& \cong \kappa(DX,\Delta')\otimes\kappa(DX,\Delta'')\\&\cong \kappa(DX,DX)\qquad \text{\cite[Remark~2.25]{cyclic}}\\&\cong \cat{A}(D^2X,DX)\\&\cong \cat{A}(K,DX \otimes DX)\\& \cong \cat{A}(D(DX\otimes DX),I) \\ &\cong \cat{A}(X\odot X,I) \ . \end{align}
		From this, we can deduce that the description given in Definition~\ref{defnondeg2pairing} is equivalent the one given in the Lemma. 	
	\end{proof}
	
With Lemma~\ref{lemmaequivdes}, 
we are in the position to characterize cyclic associative algebras relative to cyclic associative algebras in $\Lexf$, i.e.\ cyclic associative algebras in pivotal Grothendieck-Verdier categories. 
	Needless to say, we expect some sort of Frobenius algebras, but we should highlight that the fact that these are the type of algebras that the cyclic microcosm principle produces is a statement that requires a proof.
	Viewed differently, it is of course also a test that the cyclic microcosm principle in this article is set up the way it should be.
	
	Symmetric Frobenius algebras can be defined in pivotal rigid monoidal categories~\cite{fuchsstigner}; the fact that they can be defined through the microcosm principle in $\star$-autonomous categories is mentioned in~\cite[Section~2]{nlab:microcosm_principle} (see however the warning in Remark~\ref{remwarning} below). The notion also appears in~\cite[Definition~2.3.2]{egger} and \cite[Section~4.1]{fsswfrobenius}.
	For us, the following notion will turn out to be the `correct' one
	(where `correct' has the precise meaning that Proposition~\ref{propsymfa}, that we will state afterwards, holds):

	\begin{definition}\label{defsymFA}
		A \emph{symmetric Frobenius algebra} in a pivotal Grothendieck-Verdier category $\cat{A}$ in $\Lexf$ is an 
		object $F \in \cat{A}$
		together with \begin{itemize}
			\item[(M)] a multiplication $\mu : F \odot F \to F$ that is associative with respect to the associators of $(\cat{A},\odot)$,
			\item[(U)] a unit $\eta : K \to F$ for $\mu$ with respect to the unitors of $(\cat{A},\odot)$
			(the domain of the unit is the dualizing object which is the monoidal unit of $\odot$),
			\item[(P)] a non-degenerate symmetric pairing
			$\beta : F \odot F \to I$ as in Lemma~\ref{lemmaequivdes},
			\item[(I)] subject to the invariance condition on $\beta$ that $\beta (\eta,\mu)=\beta$ as maps $F \odot F \to I$. 
		\end{itemize}	
	\end{definition}

	\begin{proposition}\label{propsymfa}
		Cyclic associative algebras in cyclic associative algebras in $\Lexf$ are equivalent to symmetric Frobenius algebras in pivotal Grothendieck-Verdier categories in $\Lexf$. 
	\end{proposition}
	
	\begin{proof}
		A cyclic associative algebra with coefficients in a pivotal Grothendieck-Verdier category $\cat{A}$ has an underlying self-dual object $F \in \cat{A}$.
		This self-duality amounts to a symmetric non-degenerate pairing $\beta : F \odot F \to I$, see Lemma~\ref{lemmaequivdes}; this is point~(P) in Definition~\ref{defsymFA}.
		We now need to read off what a cyclic associative structure on $F$ with coefficients in $\cat{A}$, i.e.\
		a monoidal natural transformation $k \to \vb_F^\cat{A}$, amounts to. 
		From Remark~\ref{remgenopcomp} (or rather its cyclic version, see also
		Remark~\ref{remcyclicalg}), it follows that this can be extracted from
		a presentation of the cyclic associative operad of $\As$  in terms of generators and relations in the sense of \cite[Section~3.2 for the notion \& Theorem~4.2 for the case of $\As$]{cyclic}. After spelling this out, we obtain: \begin{itemize}
			\item An operation on $F$ for every generating operation of $\As$. The generating operations are the binary product and the unit. For the binary product, we get by definition a linear map from $k$ to
			\begin{align}
			\label{eqnkappaF}	\kappa(F,F\otimes F)\cong \cat{A}(DF,F\otimes F)\cong \cat{A}(D(F\otimes F),F)\cong \cat{A}(DF\odot DF,F)\cong \cat{A}(F \odot F,F) \ ,
			\end{align}
			i.e.\ 
			a vector in $\cat{A}(F \odot F,F)$, which is a map $\mu :F \odot F \to F$ (point~(M)). 
			For the unit, we 
			obtain a map $\eta :K\to F$ (point~(U)). 
			
			\item A relation for each of the generating isomorphisms (associator, unitor and the isomorphism $\gamma:\kappa(I,-\otimes-)\cong \kappa$ --- the `invarianciator' for the pairing~\cite[Definition~4.1~(Z)]{cyclic}). 
			The relation coming from the associator is associativity of $\mu$ with respect to the associators of $(\cat{A},\odot)$
			while the unitor gives us unitality of $\mu$ with unit $\eta$ with respect to the unitors of $(\cat{A},\odot)$. Finally, the invarianciator $\gamma$ gives us point~(I).
		\end{itemize}
		In summary, a cyclic associative algebra with coefficients in a $\Lexf$-valued cyclic associative algebra amounts precisely to the structure and relations listed in Definition~\ref{defsymFA} for a symmetric Frobenius algebra.
	\end{proof}
	
	\begin{remark}
		\label{remwarning}
		In~\cite[Proposition~3.2]{street} Grothendieck-Verdier categories (not necessarily linear ones) are described as Frobenius pseudo-monoids, so one might think that Proposition~\ref{propsymfa} (at least without the symmetry) can be seen as the result of applying the microcosm principle along the lines of~\cite[Section~2]{nlab:microcosm_principle} to Grothendieck-Verdier categories.  
		As plausible as this may sound, this is not the case! 
		The notion of Frobenius pseudo-monoid in~\cite{street} does not agree with the notion of a cyclic associative algebra in every symmetric monoidal bicategory. Most importantly, it is wrong that cyclic associative algebras in $\Cat$ are pivotal Grothendieck-Verdier categories. 
		\end{remark}

	\begin{example}\label{examplesymfrob}
		Let $\cat{C}$ be a pivotal finite tensor category (this entails in particular a rigid duality), then a source for
		symmetric Frobenius algebras are \emph{pivotal module categories}~\cite[Section~3.6]{relserre}:
		A pivotal module category $\cat{M}$ over $\cat{C}$ is a certain kind of 
		 module category whose internal hom
		$\inthom (-,-):\cat{M}^\op \boxtimes\cat{M}\to \cat{C}$  
		comes equipped with isomorphisms $\inthom(m,n)^\vee \cong \inthom(n,m)$ for all $m,n \in \cat{M}$
		 subject to further conditions. Then the internal endomorphism algebra $\inthom(m,m)$ for any $m\in \cat{M}$ is a symmetric  Frobenius algebra in $\cat{C}$~\cite[Theorem~3.14]{relserre}. 
		 A generalization of this type of construction procedure of Frobenius algebras beyond the rigid case is given in~\cite[Theorem~5.20]{fsswfrobenius}. 
		\end{example}

	\part{Applications in quantum topology and conformal field theory}
	\section{Classification of open correlators\label{secopen}}
We now turn to the applications in \emph{open conformal field theory}. We denote by $\O$ the modular \emph{open surface operad}, 
a groupoid-valued operad whose groupoid of operations of arity $n$ (to be thought of as $n$ inputs and one output)
has as objects connected compact oriented surfaces with at least one boundary component and $n+1$ parametrized intervals embedded in its boundary. Morphisms of $\O(n)$ are mapping classes of diffeomorphisms between these surfaces
 preserving the orientation and the boundary parametrization.
	References for a more detailed definition of $\O$ include~\cite{costellotcft,giansiracusa}; here we follow the definitions in~\cite[Section~3]{sn},
	where the open-closed case is considered, and also~\cite{envas}. 
	
	\begin{definition}
		Let $\omf$ be a $\Lexf$-valued
		modular $\O$-algebra, i.e.\ an open modular functor with values in $\Lexf$, also known as categorified open topological field theory~\cite[Section~2]{envas} (this is the monodromy data of an open conformal field theory).
		A \emph{consistent system of open correlators} for $\omf$ is a modular $\O$-algebra with coefficients in $\omf$. 
		\end{definition}
	
This is in line with the notion of consistent systems of correlators in~\cite{jfcs,jfcslog} (that are also defined via monoidal natural transformations --- so the connection is not surprising), but here we consider the open version. Let us unpack the data:
	Denote by $\cat{A}$ the underlying category  for the open modular functor $\omf$ (one might call this the \emph{category of boundary labels}). For a corolla $T$, $\Sigma \in \O(T)$ and $X_i \in \cat{A}$ for $1\le i\le n:=|\Legs(T)|$, the vector space 
	$\omf(\Sigma;X_1,\dots,X_n)$ is the space of conformal blocks for the surface $\Sigma$ and boundary labels $X_1,\dots,X_n$.
	A consistent system of open correlators $\xi \in \catf{ModAlg}(\O;\omf)$ has an underlying object $F\in \cat{A}$ and vectors $\xi_\Sigma^F \in \omf(\Sigma;F,\dots,F)$ in the spaces of conformal blocks that are invariant under the mapping class group actions and compatible with sewing along boundary intervals.

		\begin{corollary}\label{corext-o}
		For any cyclic associative 
		algebra $\cat{A}$ in $\cat{S}$, there is an equivalence
		\begin{equation}\label{eqnequivcycframed-o}
			\begin{tikzcd}
				\catf{CycAlg}(\As;\cat{A}) \ar[rrrr, shift left=2,"\text{extension}"] &&\simeq&& \ar[llll, shift left=2,"\text{restriction}"] 
				\catf{ModAlg}\left( \O; \Ao\right) \ , 
			\end{tikzcd}
		\end{equation} 
	where $\Ao$ denotes the modular extension of the cyclic $\As$-algebra $\cat{A}$.
	\end{corollary}
	
	\begin{proof}
		This follows directly
		 from Theorem~\ref{exttheorem} if we additionally use the equivalence
		$\Pi |B\Envint \As| \ra{\simeq} \O$ from~\cite[Equation~(2.1)]{envas}. 
		\end{proof}

		\begin{theorem}[Classification of open correlators]\label{thmopencorrelators}
				Given an open modular functor $\omf$ in $\Lexf$, let $\cat{A}$ be the pivotal Grothendieck-Verdier category obtained by evaluation of $\omf$
			on disks with intervals 
			embedded
			in its boundary.
			Then the consistent systems of
			open correlators for $\omf$ are exactly symmetric Frobenius algebras in $\cat{A}$.
			More explicitly,
			the open correlator $\xi^F$
			associated
			to 	a symmetric Frobenius algebra  $F$ in $\cat{A}$
			amounts to
			vectors $\xi_\Sigma^F \in \omf(\Sigma;F,\dots,F)$ in the spaces of conformal blocks associated
			to  surfaces $\Sigma \in \O(n)$ with $F$ appearing $n$ times as the boundary label. 
			These vectors are mapping class group invariant and solve the constraints for sewing along intervals.
	\end{theorem}

	\begin{proof}
		Every open modular functor $\omf$
		 is equivalent to the modular extension of its restriction
		 to disks with marked intervals, which is a cyclic $\As$-algebra
		  $\cat{A}$, i.e.\ a pivotal Grothendieck-Verdier category, see~\cite[Theorem~2.2]{envas}.
		This implies that the consistent systems of open correlators for $\omf$ are described by the category
		$\catf{ModAlg}\left( \O; \Ao\right)$. Now the statement follows from
		\begin{align}
			\catf{ModAlg}\left( \O; \Ao\right) \stackrel{\text{Corollary~\ref{corext-o}}}{\simeq}\catf{CycAlg}\left( \As; \cat{A}\right) 
			\stackrel{\text{Proposition~\ref{propsymfa}}}{\simeq} \mathsf{symmetric\ Frobenius\ algebras\ in} \ \cat{A} \ . 
			\end{align}
		\end{proof}

	\section{A calculation recipe  for open correlators\label{seccalc}}
	In this section, we demonstrate how the open correlators from Theorem~\ref{thmopencorrelators} can be calculated explicitly.
	Suppose that $\omf$ is an open modular functor and $\cat{A}$ its underlying pivotal Grothendieck-Verdier category, i.e.\ the restriction to disks $\mathbb{D}^2_n$ with intervals $n\ge 1$ in its boundary. 
	By \cite[Theorem~2.2]{envas} $\omf$ is, up to a canonical equivalence, the modular extension of $\cat{A}$, i.e.\ $\omf\simeq \Ao$ as open modular functors
	($\simeq$ is an equivalence of modular algebras over the open surface operad).
	For this reason, we know~\cite[Example~7.3]{cyclic}
	\begin{align}
		\omf(    \mathbb{D}^2_n ;     X_1,\dots,X_n  ) \cong \cat{A}(K,X_1\otimes \dots \otimes X_1)  \label{eqndiskAhat}
		\end{align}
	for   $X_1,\dots,X_n \in \cat{A}$. Let us now describe in more detail
	the consistent system of open correlators for a symmetric Frobenius algebra $F$ in $\cat{A}$:
	\begin{itemize}
\item \emph{The multiplication, the unit and the pairing:}	For $n=3$ and since $F$ is self-dual
	\begin{align}
			\omf(    \mathbb{D}^2_3 ;      F , F, F  ) \cong \cat{A}(K, F \otimes F \otimes F) \cong \kappa(F,F\otimes F) \stackrel{\eqref{eqnkappaF}}{\cong} \cat{A}(F \odot F, F) \ . 
		\end{align}
	The consistent system of open correlator $\xi^F$ associated to a symmetric Frobenius algebra gives us a vector $\xi_{\mathbb{D}^2_3}^F \in \omf(    \mathbb{D}^2_3 ;      F , F, F  )$, and by definition this is the multiplication $\mu : F \odot F \to F$. 
	Pictorially, this means that the disk with three intervals in its boundary, each labeled with $F$,
	 is `decorated' with the multiplication operation:
		\begin{equation}
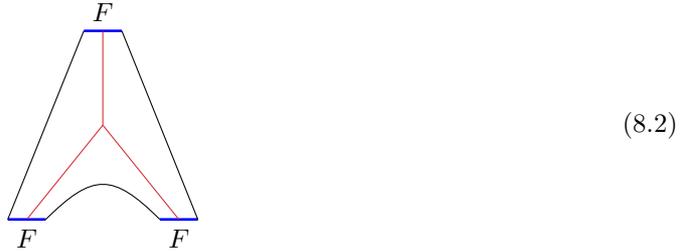
	\begin{array}{c}	
	\begin{tikzpicture}[scale=0.5]
		\begin{pgfonlayer}{nodelayer}
			\node [style=none] (0) at (-7, 5) {};
			\node [style=none] (1) at (-6, 5) {};
			\node [style=none] (2) at (-9, 0) {};
			\node [style=none] (3) at (-8, 0) {};
			\node [style=none] (4) at (-5, 0) {};
			\node [style=none] (5) at (-4, 0) {};
			\node [style=none] (10) at (-6.5, 2.5) {};
			\node [style=none] (11) at (-4.5, 0) {};
			\node [style=none] (14) at (-6.5, 5) {};
			\node [style=none] (15) at (-8.5, 0) {};
			\node [style=none] (18) at (-4.5, -0.5) {$F$};
			\node [style=none] (21) at (-6.5, 5.5) {$F$};
			\node [style=none] (22) at (-8.5, -0.5) {$F$};
			\node [style=wbox] (34) at (-11, 4.75) {$\mu$};
			\node [style=none] (35) at (-6.75, 2.75) {};
			\node [style=none] (36) at (-10.5, 4.75) {};
		\end{pgfonlayer}
		\begin{pgfonlayer}{edgelayer}
			\draw [style=open] (0.center) to (1.center);
			\draw [style=open] (5.center) to (4.center);
			\draw [style=open] (3.center) to (2.center);
			\draw (2.center) to (0.center);
			\draw [bend left=45, looseness=1.50] (3.center) to (4.center);
			\draw [style=RED] (11.center) to (10.center);
			\draw [style=RED] (10.center) to (15.center);
			\draw (5.center) to (1.center);
			\draw [style=RED] (10.center) to (14.center);
			\draw [style=end arrow, bend left] (36.center) to (35.center);
		\end{pgfonlayer}
	\end{tikzpicture}
\end{array} \label{eqndefondisks3}
	\end{equation}
	The boundary intervals are drawn in blue.
	The decoration is the red graph whose legs are colored with $F$; the multiplication operation should be imagined as sitting at the trivalent vertex. Note however that the legs are not incoming or outgoing. It is the advantage of working with cyclic operads and their algebras that we do not (and should not) make such distinctions. 
	For $n=1$, $\omf(    \mathbb{D}^2_1 ;      F   )\cong \cat{A}(K,F)$, and the distinguished vector is the unit $K\to F$. For $n=2$, $\omf(    \mathbb{D}^2_2 ;      F  ,F )\cong \cat{A}(K,F\otimes F)$, and the distinguished vector
	is the copairing $\delta : K \to F \otimes F$ by the unitality requirement in Definition~\ref{defmodalgcoeff}. Since this copairing is part of a self-duality, it is by construction non-degenerate. This is simply because the non-degeneracy of the pairing on the microcosmic level needs to match the non-degeneracy of the macrocosmic one. In~\cite[Section~4.5]{ffrsunique} the non-degeneracy is an extra assumption called \emph{non-degeneracy of the disk two-point function}.  
	
	Thanks to the self-duality of $F$, we can see $D\eta$ as a morphism $F\to I$, the \emph{counit} of $F$. Its evaluation on the unit $K\to F$ is the distinguished vector in the space of conformal blocks $\omf(\mathbb{D}^2_0)\cong \cat{A}(K,I)$ for the disk without marked intervals.
	
	\item \emph{A graphical presentation via Poincaré duality:}
	We can draw~\eqref{eqndefondisks3} as a triangle, with the embedded intervals as segments on the edges. The decoration with the multiplication then describes a partition of the triangle that is \emph{Poincaré dual} to the original one:
		\begin{equation}
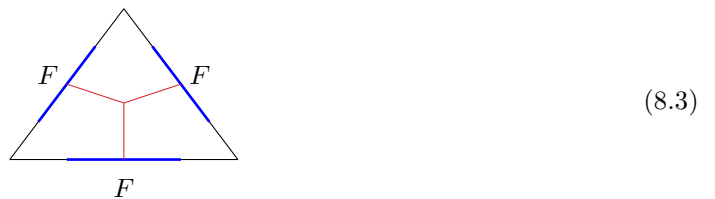
	\begin{array}{c}\begin{tikzpicture}[scale=0.5]
			\begin{pgfonlayer}{nodelayer}
				\node [style=none] (18) at (5, 2.25) {$F$};
				\node [style=none] (21) at (1, 2.25) {$F$};
				\node [style=none] (22) at (3, -0.75) {$F$};
				\node [style=none] (23) at (0, 0) {};
				\node [style=none] (24) at (3, 4) {};
				\node [style=none] (25) at (6, 0) {};
				\node [style=none] (26) at (0.75, 1) {};
				\node [style=none] (27) at (2.25, 3) {};
				\node [style=none] (28) at (1.5, 0) {};
				\node [style=none] (29) at (4.5, 0) {};
				\node [style=none] (30) at (3.75, 3) {};
				\node [style=none] (31) at (5.25, 1) {};
				\node [style=none] (32) at (3, 1.5) {};
				\node [style=none] (33) at (1.5, 2) {};
				\node [style=none] (34) at (4.5, 2) {};
				\node [style=none] (35) at (3, 0) {};
			\end{pgfonlayer}
			\begin{pgfonlayer}{edgelayer}
				\draw (23.center) to (25.center);
				\draw (25.center) to (24.center);
				\draw (24.center) to (23.center);
				\draw [style=open] (30.center) to (31.center);
				\draw [style=open] (27.center) to (26.center);
				\draw [style=open] (28.center) to (29.center);
				\draw [style=RED] (32.center) to (33.center);
				\draw [style=RED] (32.center) to (34.center);
				\draw [style=RED] (32.center) to (35.center);
			\end{pgfonlayer}
		\end{tikzpicture}
	\end{array} \label{eqndefondisks3t}
	\end{equation}
	In other words, we recover the labeling principle underlying the constructions in \cite{frs1,frs2,frs3,frs4,ffrs}.

	\item
	\emph{Disks with more boundary intervals:}	
	For $n\ge 4$, the vector $\xi_{\mathbb{D}^2_n}^F$ is obtained by gluing disks with three intervals on their boundary together; pictorially: 
	\begin{equation}
	\begin{array}{c}\begin{tikzpicture}[scale=0.5]
			\begin{pgfonlayer}{nodelayer}
				\node [style=none] (0) at (-5, 5) {};
				\node [style=none] (1) at (-4, 5) {};
				\node [style=none] (2) at (-8, 4) {};
				\node [style=none] (3) at (-8, 3) {};
				\node [style=none] (4) at (-7, 0) {};
				\node [style=none] (5) at (-6, 0) {};
				\node [style=none] (6) at (-2, 1) {};
				\node [style=none] (7) at (-1, 1) {};
				\node [style=none] (8) at (0, 3) {};
				\node [style=none] (9) at (0, 4) {};
				\node [style=none] (10) at (-5.25, 2.25) {};
				\node [style=none] (11) at (-6.5, 0) {};
				\node [style=none] (12) at (-1.5, 1) {};
				\node [style=none] (13) at (0, 3.5) {};
				\node [style=none] (14) at (-4.5, 5) {};
				\node [style=none] (15) at (-8, 3.5) {};
				\node [style=none] (16) at (-2, 3) {};
				\node [style=none] (17) at (-4, 3) {};
				\node [style=none] (18) at (-6.5, -0.5) {$F$};
				\node [style=none] (19) at (-1.5, 0.5) {$F$};
				\node [style=none] (20) at (0.5, 3.5) {$F$};
				\node [style=none] (21) at (-4.5, 5.5) {$F$};
				\node [style=none] (22) at (-8.5, 3.5) {$F$};
				\node [style=none] (23) at (10.75, 3.5) {$\xi_{\mathbb{D}^2_5}^F \in   \cat{A}(K,F ^{\otimes 5})$};
				\node [style=none] (24) at (3, 3.5) {$\leadsto$};
				\node [style=none] (25) at (-3.25, 1.75) {};
				\node [style=none] (26) at (-2.5, 4) {};
				\node [style=none] (27) at (-5.75, 4.25) {};
				\node [style=none] (28) at (-4.5, 1.75) {};
			\end{pgfonlayer}
			\begin{pgfonlayer}{edgelayer}
				\draw [style=open] (0.center) to (1.center);
				\draw [style=open] (9.center) to (8.center);
				\draw [style=open] (6.center) to (7.center);
				\draw [style=open] (5.center) to (4.center);
				\draw [style=open] (3.center) to (2.center);
				\draw [bend right, looseness=1.25] (2.center) to (0.center);
				\draw [bend left=45, looseness=1.50] (3.center) to (4.center);
				\draw [bend left=60, looseness=1.25] (5.center) to (6.center);
				\draw [bend left=45, looseness=1.25] (7.center) to (8.center);
				\draw [bend left=45, looseness=0.75] (9.center) to (1.center);
				\draw [style=RED] (11.center) to (10.center);
				\draw [style=RED] (12.center) to (16.center);
				\draw [style=RED] (16.center) to (13.center);
				\draw [style=RED] (16.center) to (17.center);
				\draw [style=RED] (17.center) to (14.center);
				\draw [style=RED] (10.center) to (15.center);
				\draw [style=RED] (10.center) to (17.center);
				\draw [style=mydots] (27.center) to (28.center);
				\draw [style=mydots] (25.center) to (26.center);
			\end{pgfonlayer}
		\end{tikzpicture}
	\end{array} \label{eqndefondisks}
		\end{equation}
	The gluing is along the gray dashed lines.
	Under the isomorphism \begin{align}
		\cat{A}(K,F ^{\otimes n})\cong \cat{A}(DF,F^{\otimes (n-1)})\cong \cat{A}((DF)^{\odot (n-1)},F)\cong \cat{A}(F^{\odot (n-1)},F) \ ,  \end{align}
	the vector $\xi_{\mathbb{D}^2_5}^F$ is the arity $n-1$ multiplication $F^{\odot (n-1)}\to F$ that multiplies elements together from left to right.
		If we think of this surface
		 as obtained by gluing three triangles~\eqref{eqndefondisks3t} together, the decoration is again obtained by passing to the Poincaré dual. 
		
			\item \emph{More complicated surfaces:}
			Beyond disks with marked intervals,
			 $\xi$ is completely determined by the fact that  $\Ao$ is a modular algebra. The detailed treatment is in~\cite[Example~7.3]{cyclic}, see also~\cite[eq.~(4.1)]{envas}; we discuss here an example:
			 Suppose that we glue in~\eqref{eqndefondisks} the two rightmost boundary intervals together, thereby creating an annulus $\Sigma$ with three intervals embedded in its outer boundary circle:
			 	\begin{equation}
			 	\begin{array}{c}\begin{tikzpicture}[scale=0.5]
			 			\begin{pgfonlayer}{nodelayer}
			 				\node [style=none] (0) at (-5, 5) {};
			 				\node [style=none] (1) at (-4, 5) {};
			 				\node [style=none] (2) at (-8, 4) {};
			 				\node [style=none] (3) at (-8, 3) {};
			 				\node [style=none] (4) at (-7, 0) {};
			 				\node [style=none] (5) at (-6, 0) {};
			 				\node [style=none] (6) at (-2, 1) {};
			 				\node [style=none] (7) at (-1, 1) {};
			 				\node [style=none] (8) at (0, 3) {};
			 				\node [style=none] (9) at (0, 4) {};
			 				\node [style=none] (10) at (-5.25, 2.25) {};
			 				\node [style=none] (11) at (-6.5, 0) {};
			 				\node [style=none] (12) at (-1.5, 1) {};
			 				\node [style=none] (13) at (0, 3.5) {};
			 				\node [style=none] (14) at (-4.5, 5) {};
			 				\node [style=none] (15) at (-8, 3.5) {};
			 				\node [style=none] (16) at (-2, 3) {};
			 				\node [style=none] (17) at (-4, 3) {};
			 				\node [style=none] (18) at (-6.5, -0.5) {$F$};
			 				\node [style=none] (19) at (-1.5, 0.5) {$F$};
			 				\node [style=none] (22) at (-8.5, 3.5) {$F$};
			 				\node [style=none] (23) at (10.75, 3.5) {$\xi_{\mathbb{D}^2_5}^F \in   \cat{A}(K,F ^{\otimes 3}\otimes\mathbb{F})$};
			 				\node [style=none] (24) at (3, 3.5) {$\leadsto$};
			 				\node [style=none] (25) at (-3.25, 1.75) {};
			 				\node [style=none] (26) at (-2.5, 4) {};
			 				\node [style=none] (27) at (-5.75, 4.25) {};
			 				\node [style=none] (28) at (-4.5, 1.75) {};
			 				\node [style=none] (29) at (-1, 6.25) {};
			 				\node [style=none] (30) at (-1.5, 5.25) {};
			 				\node [style=none] (31) at (-0.75, 6.75) {};
			 			\end{pgfonlayer}
			 			\begin{pgfonlayer}{edgelayer}
			 				\draw [style=open] (6.center) to (7.center);
			 				\draw [style=open] (5.center) to (4.center);
			 				\draw [style=open] (3.center) to (2.center);
			 				\draw [bend right, looseness=1.25] (2.center) to (0.center);
			 				\draw [bend left=45, looseness=1.50] (3.center) to (4.center);
			 				\draw [bend left=60, looseness=1.25] (5.center) to (6.center);
			 				\draw [bend left=45, looseness=1.25] (7.center) to (8.center);
			 				\draw [bend left=45, looseness=0.75] (9.center) to (1.center);
			 				\draw [style=RED] (11.center) to (10.center);
			 				\draw [style=RED] (12.center) to (16.center);
			 				\draw [style=RED] (16.center) to (13.center);
			 				\draw [style=RED] (16.center) to (17.center);
			 				\draw [style=RED] (17.center) to (14.center);
			 				\draw [style=RED] (10.center) to (15.center);
			 				\draw [style=RED] (10.center) to (17.center);
			 				\draw [style=mydots] (27.center) to (28.center);
			 				\draw [style=mydots] (25.center) to (26.center);
			 				\draw [bend left=105, looseness=0.75] (1.center) to (9.center);
			 				\draw [in=30, out=60, looseness=2.25] (0.center) to (8.center);
			 				\draw [style=REDarrow, in=165, out=105, looseness=0.75] (14.center) to (29.center);
			 				\draw [style=REDarrow, in=-15, out=15, looseness=1.50] (13.center) to (29.center);
			 				\draw [style=mydots] (30.center) to (31.center);
			 			\end{pgfonlayer}
			 		\end{tikzpicture}
			 	\end{array} \label{eqndefondisks2}
			 \end{equation}
		 \end{itemize} 
			 The fact that  $\Ao$ is a modular algebra, in particular the first square in
			 \cite[Definition~2.13]{cyclic}, tells us how to obtain $	\omf(\Sigma;X_1,X_2,X_3)$ from this triangulation of $\Sigma$: We need to fill  the coevaluation object $\Delta =\int^{X \in \cat{A}} DX \boxtimes X \in \cat{A}\boxtimes\cat{A}$ into the arguments of the functor $\Ao(\mathbb{D}_5^2; -)$ in which the gluing occurs. This gives us
			 \begin{align}
			\omf(\Sigma;X_1,X_2,X_3) \cong		\Ao(\Sigma;X_1,X_2,X_3) \cong \Ao(\mathbb{D}_5^2; X_1 , X_2,X_3,\Delta) \ .
			 	\end{align}
		 	With~\eqref{eqndiskAhat}, we arrive at 
		 	\begin{align}
		 		\omf(\Sigma;X_1,X_2,X_3)&\cong	\Ao(\Sigma;X_1,X_2,X_3) \cong \cat{A}(K,X_1\otimes X_2 \otimes X_3 \otimes \mathbb{F}) \\ \text{with}\quad \mathbb{F}& := \otimes \left( \int^{X \in \cat{A}} DX\boxtimes X \right)\ . 
 		 		\end{align}
 	 		Since the consistent system of open correlators $\xi^F$ solves the sewing constraints for gluing along intervals, $\xi_\Sigma^F$ is the image of
 	 		$\xi_{\mathbb{D}_5^2}^F$ under\small
 	 	\begin{align}	\omf(    \mathbb{D}^2_5 ;      F^{\boxtimes 5} )\cong \cat{A}(K,F^{\otimes 5}) \ra{\text{self-duality of $F$}} 
 	 		\cat{A}(K,F^{\otimes 3}\otimes DF \otimes F) \to \cat{A}(K,F^{\otimes 3} \otimes \mathbb{F})\cong \omf(\Sigma;F^{\boxtimes 3}) \ , 
		\end{align}\normalsize
	where the third map is induced by the map 
	\begin{align}\label{eqnsorting}
		 DF \otimes F \to \mathbb{F}\end{align} obtained by applying $\otimes$ to the structure map 
	$
		DF \boxtimes F \to \int^{X \in \cat{A}} DX\boxtimes X
		$ of the coend. Note that by definition the map $F \boxtimes F \to \int^{X \in \cat{A}}DX\boxtimes X$ obtained from the self-duality $DF\cong F$ and structure map of the coend is exactly the symmetric 2-pairing (this is how we defined self-duality in the first place, see Definition~\ref{defnondeg2pairing}).
		We can symbolically represent this vector through the red graph in \eqref{eqndefondisks2}.
		The two arrows landing on the dashed cut symbolize how the two copies of $F$ are mapped to the coend through the structure map and the self-duality.

\begin{example}
	As we have seen above,
	the open correlators can be calculated directly from the operations of the symmetric Frobenius algebra $F\in \cat{A}$ and the  map~\eqref{eqnsorting}
	 `sorting $F \otimes F$ into the coend'. In Hopf-algebraic 
	  examples, the latter part can be made more explicit: Suppose that $H$ is a pivotal finite-dimensional Hopf algebra, and $\cat{A}$ the category of finite-dimensional $H$-modules.
	Then $\mathbb{F}=\int^{X \in \cat{A}} X^\vee \otimes X$ (because $\otimes$ is exact), and we can make the following further statements that follow from
	\cite[Theorem~7.4.13]{kl}: The object $\mathbb{F}$ 
	 is given by the coadjoint representation $H^*_\text{coadj}$, which is the dual of the adjoint $H$-action on itself given by $h.x := h'x S(h'')$, where we have used the Sweedler notation for the coproduct of $H$. For a symmetric Frobenius algebra $F$ in $H$-modules, the map $F\otimes F \to H_\text{coadj}^*$ from
	  \eqref{eqnsorting} is given by
	  \begin{align}
	  F \otimes F \ni	x \otimes y \mapsto  \left( H\ni h \mapsto \psi(x)(h.y)  \right) \ , 
	  	\end{align} with the self-duality $\psi :F \to F^*$. 
	\end{example}

		\begin{remark}[Relation to the string-net construction of correlators in rational conformal field theory]\label{remsn}
			Let $\cat{C}$ be a pivotal finite tensor category. Then there is a string-net model $\SNC$ for the open-closed modular functor for the Drinfeld center $Z(\cat{C})$~\cite{sn}. The open part of $\SNC$ is equivalent to $\widehat{\cat{C}}$~\cite[Theorem~4.2]{envas}. If $\cat{C}$ is a modular fusion category, consistent systems of correlators for the modular functor $Z(\cat{C})$ have been constructed in~\cite{rcftsn} from special symmetric Frobenius algebras in $\cat{C}$ using string-net techniques. Theorem~\ref{thmopencorrelators} can be seen as a non-semisimple generalization of the open part of this construction (which does not need specialness because we have not discussed yet extensions beyond the open sector). Nonetheless, we should highlight that even though  Theorem~\ref{thmopencorrelators}, at least for pivotal finite tensor categories, should admit a string-net interpretation through the abstract comparison result \cite[Theorem~4.2]{envas}, it is largely unclear how this string-net picture for the open correlators in the non-semisimple situation would look like. One might see this as a selling point for the modular microcosm principle because it shows that it allows for constructions that otherwise are  not available in this generality. On the other hand, we lose, in comparison to the semisimple string-net techniques in~\cite{rcftsn} quite a bit of explicitness in the calculations because powerful ribbon graph techniques enter as a `black box'. 
			\end{remark}

	\spaceplease	
	\section{Classification of ansular correlators\label{secansular}}
	By \cite[Theorem~5.13]{cyclic} $\Lex$-valued
	cyclic algebras over the $\framed$-operad are equivalent to ribbon Grothendieck-Verdier categories as defined in~\cite{bd}, i.e.\ braided Grothendieck-Verdier categories $\cat{A}$ equipped with a natural automorphism $\theta_X:X\to X$ called \emph{balancing} such that $\theta_I=\id_I$, $\theta_{X \otimes Y}=c_{Y,X}c_{X,Y}(\theta_X \otimes \theta_Y)$ for $X,Y \in \cat{A}$, with $c_{X,Y}:X\otimes Y \ra{\cong} Y \otimes X$ being the braiding, such that additionally the compatibility $D\theta_X=\theta_{DX}$ with the Grothendieck-Verdier duality holds. As explained in~\cite{bd}, any ribbon Grothendieck-Verdier category has an underlying pivotal Grothendieck-Verdier category, i.e.\ it is an enhancement of the structure discussed in the previous two sections.
	If the Grothendieck-Verdier duality is rigid and the unit simple, we obtain the notion of a \emph{finite ribbon category}~\cite{egno}. 
	
	\begin{definition}\label{defcommsymFA}
	Let $\cat{A}$ be a ribbon Grothendieck-Verdier category in $\Lexf$. We define a \emph{symmetric commutative Frobenius algebra} in $\cat{A}$ as a symmetric Frobenius algebra in the underlying pivotal Grothendieck-Verdier category which is also 
		commutative with respect to the braiding.
\end{definition}
	
	\begin{proposition}\label{propribbonfrob}
		Cyclic framed $E_2$-algebras in cyclic framed $E_2$-algebras in $\Lexf$ are equivalent to symmetric commutative Frobenius algebras in ribbon Grothendieck-Verdier categories.
	\end{proposition}

	\begin{proof}
		Cyclic framed $E_2$-algebras in $\Lexf$ are ribbon Grothendieck-Verdier categories in $\Lexf$ by \cite[Theorem~5.13]{cyclic}. This means that in comparison to 
		the situation of Proposition~\ref{propsymfa} we have an additional generating isomorphism, namely the braiding. This gives us one additional relation for cyclic framed $E_2$-algebras in a ribbon Grothendieck-Verdier category, namely the braided commutativity in Definition~\ref{defcommsymFA}. A priori, we have another generating isomorphism, namely the balancing, but a ribbon Grothendieck-Verdier category can be described as a braided Grothendieck-Verdier category with pivotal structure that is subject to a property (this is \cite[Corollary~8.3]{bd}, see also \cite[Lemma~5.10]{cyclic}). Therefore, there is no additional relation through the balancing.
	\end{proof}

	\begin{remark}\label{remtrivialtwist}
		Even though it was implicitly part of the proof of Proposition~\ref{propribbonfrob}, let us mention explicitly:
		For a symmetric commutative Frobenius algebra $F$ in a ribbon Grothendieck-Verdier category, the triviality of the balancing ($\theta_F=\id_F$) is automatic. This is a generalization of~\cite[Proposition~2.25]{correspondences}.
		\end{remark}

	By replacing in the definition of the surface operad all surfaces with three-dimensional handlebodies (as always compact oriented; instead of boundary components, we consider disks embedded in the boundary surface of the handlebody), one obtains the modular operad $\Hbdy$ of handlebodies, see \cite[Section~4.3]{giansiracusa} and \cite[Section~7.2]{cyclic} for more details. A ($\Lexf$-valued) $\Hbdy$-algebra is called \emph{ansular functor}, and by the main result of~\cite{mwansular} building on~\cite{giansiracusa,cyclic,mwdiff} the restriction to genus zero produces an equivalence to cyclic framed $E_2$-algebras:
	\begin{equation}\label{eqnequivcycframed0}
		\begin{tikzcd}
			\catf{CycAlg}(\framed) \ar[rrrr, shift left=2,"\text{modular extension}"] &&\simeq&& \ar[llll, shift left=2,"\text{restriction}"] 
			\catf{ModAlg}\left( \Hbdy \right) \quad \text{for \emph{bicategorical} algebras.} 
		\end{tikzcd}
	\end{equation}

		\begin{definition}[Ansular correlator]
		Let $\mathfrak{B}$ be an 
		ansular functor with values in $\Lexf$. 
		A \emph{consistent system of ansular correlators} for $\mathfrak{B}$ is a modular $\Hbdy$-algebra with coefficients in the ansular functor $\mathfrak{B}$.
	\end{definition}

	\begin{corollary}\label{corext}
		For any cyclic framed $E_2$-algebra $\cat{A}$ in $\cat{S}$, there is an equivalence
		\begin{equation}\label{eqnequivcycframed}
			\begin{tikzcd}
				\catf{CycAlg}(\framed;\cat{A}) \ar[rrrr, shift left=2,"\text{extension}"] &&\simeq&& \ar[llll, shift left=2,"\text{restriction}"] 
				\catf{ModAlg}\left( \Hbdy; \widehat{\cat{A}}\right) \ . 
			\end{tikzcd}
		\end{equation} 
	\end{corollary}

	\begin{proof}
		The equivalence~\eqref{eqnequivcycframed} follows from  Theorem~\ref{exttheorem} 
		about micrcosmic extension
		if we additionally use 
		$\Pi |B\Envint \framed| \ra{\simeq} \Hbdy$
		from~\cite[Theorem~5.2]{mwansular}
		(building as mentioned on~\cite{giansiracusa,cyclic,mwdiff}).
	\end{proof}

		\begin{theorem}[Classification of ansular correlators]\label{thmansularcor}
		The consistent systems of
	ansular correlators for an ansular
	functor are exactly symmetric commutative
	Frobenius algebras $F\in\cat{A}$
	in the ribbon Grothendieck-Verdier category $\cat{A}$ obtained from the 
	ansular functor by genus zero restriction.
	More explicitly,
	the ansular correlator $\xi^F$
	associated
	to 	a symmetric commutative Frobenius algebra $F$ in $\cat{A}$
	amounts to
	vectors $\xi_H^F \in \widehat{\cat{A}}(H;F,\dots,F)$ in the spaces of conformal blocks associated
	to handlebodies $H$
	that are handlebody group invariant and solve the constraints for sewing along disks embedded in the boundary surface of the handlebodies.
	\end{theorem}

	\begin{proof}
		The proof is by assembly of the results that we have already proven and completely analogous to the proof of Theorem~\ref{thmopencorrelators}, but this time we apply the microcosmic extension to the situation in Corollary~\ref{corext}
		 and Proposition~\ref{propribbonfrob} instead.
	\end{proof}

\begin{remark}\label{remgenuszerocorrelators}
	The cyclic framed $E_2$-operad is equivalent to the cyclic genus zero surface operad (or equivalently the cyclic genus zero handlebody operad). This is implicit in \cite{salvatorewahl,giansiracusa}, see \cite[Proposition 5.3]{cyclic} for an explicit statement for the groupoid-valued version.
	With the connection to correlators explained in the introduction
	(but this time just applied to genus zero), this
	 means that for a cyclic framed $E_2$-algebra $\cat{A}$ in a symmetric monoidal bicategory, the cyclic framed $E_2$-algebras in $\cat{A}$
are exactly genus zero correlators for $\cat{A}$ in the sense \cite[Section~4.2]{jfcs}.
Theorem~\ref{thmansularcor} now tells us that, for any cyclic framed $E_2$-algebra $\cat{A}$, genus zero correlators for $\cat{A}$
and ansular correlators for $\widehat{\cat{A}}$ are equivalent.
	\end{remark}

\begin{example}\label{examplesymcomfa}
	If $\cat{A}$ is a unimodular finite ribbon category, then $B=\int^{X \in \cat{A}} X^\vee \boxtimes X$ can be equipped with the structure of a symmetric commutative Frobenius algebra in $\bar{\cat{A}}\boxtimes\cat{A}$, where $\bar{\cat{A}}$ is $\cat{A}$ with the same monoidal product, but with braiding and balancing inverted, see~\cite[Proposition~2.3]{fuchsschweigertstignerhopf} and~\cite[Section~2]{cardycase}.
	Via the braided monoidal functor $\bar{\cat{A}}\boxtimes\cat{A} \to Z(\cat{A})$ to the Drinfeld center of $\cat{A}$~\cite[Section~4]{eno-d}, we can turn $B$ into a symmetric commutative Frobenius algebra $\mathbb{F}=\int^{X \in \cat{A}} X^\vee \otimes X$ in $Z(\cat{A})$. The underlying algebra structure was already described in~\cite[Section~2]{lyu}. The braided commutative algebra in $Z(\cat{A})$ appears in~\cite[Definition~4 \& Proposition~5]{neuchlschauenburg}. The construction of symmetric commutative Frobenius algebras in Drinfeld centers can be vastly generalized using the theory of pivotal module categories and internal endotransformations~\cite[Corollary~39]{internal}.
	\end{example}

\begin{example}
	Suppose that $\cat{A}$ is a modular category.
	Then the end $\mathbb{A}=\int_{X \in \cat{A}} X^\vee \otimes X$ becomes an algebra in $\cat{A}$ by applying component-wise the evaluations, even a Frobenius algebra, see
	\cite[Lemma~3.5]{dmno} and \cite[Theorem~6.1~(4)]{shimizuunimodular}. 
	But generally no structure of an ansular correlator on $\mathbb{A}$ exists. 
	Indeed, if $\cat{A}$ is the category of finite-dimensional modules over the small quantum group
	$\bar U_q (\mathfrak{sl}_2)$, where $q = \exp (2\pi \text{i}/r)$ with odd integer $r\ge 3$, it is shown in \cite[Example~6.5]{mwdehn} using the results of \cite[Section~5.2]{gai2} that the balancing $\theta_\mathbb{A}$ is not the identity because the action of some elements in the Johnson kernel on the spaces of conformal blocks for closed surfaces is not trivial. This implies thanks to Remark~\ref{remtrivialtwist} that $\mathbb{A}$ cannot give rise to an ansular correlator. 
	\end{example}

If  $\cat{A}$ is a finite ribbon category,
consistent systems of ansular correlators \emph{in genus zero} 
agree exactly with the consistent systems of genus zero correlators that are being classified by symmetric commutative Frobenius algebras in~\cite[Proposition~4.7]{jfcs}. 
Therefore, Theorem~\ref{thmansularcor} immediately implies:

\begin{corollary}
	Each consistent system of
	 genus zero correlators for a finite ribbon category in the sense of Fuchs-Schweigert 
	 extends uniquely to a consistent system of ansular correlators. 
	\end{corollary}
	 
	Let us briefly highlight why this is actually a non-obvious statement:
	For a handlebody $H_g$ of genus $g$,   \cite[Proposition~4.7]{jfcs} easily gives us the invariance under Dehn twists. But unlike for mapping class group of surfaces, handlebody groups are not generated by Dehn twists. Instead, the Dehn twists generate a normal subgroup $\catf{Tw}(H_g)\subset \Map(H_g)$, the \emph{twist group}, that fits into a short exact sequence
	 \begin{align}
	 	1 \to \catf{Tw}(H_g)\to\Map(H_g)\to\catf{Out}(F_g)\to 1 \ , \label{eqnluft}
	 \end{align} where $\catf{Out}(F_g)$ is the group of outer automorphisms of the free group $F_g$ on $g$ generators, see \cite{luft}. In other words, $\Map(H_g)$ is much larger than $\catf{Tw}(H_g)$.

	 In a finite ribbon category $\cat{A}$, the monoidal unit $I$ is obviously a symmetric commutative Frobenius algebra. Theorem~\ref{thmansularcor} then implies:
	
	\begin{corollary}\label{corpointing}
		Let $\cat{A}$ be a finite ribbon category.
		Then the spaces of conformal blocks $\widehat{\cat{A}}(H)$ of the ansular functor for $\cat{A}$ evaluated on 
		 handlebodies $H$ without embedded disks come
		with a distinguished 
		$\Map(H)$-invariant non-zero
		vector $\xi_H \in \widehat{\cat{A}}(H)$. 
		\end{corollary}

	\begin{proof}
		We apply Theorem~\ref{thmansularcor} to the symmetric commutative Frobenius algebra $I \in \cat{A}$ (this needs rigidity).
		The fact that the vectors $\xi_H$ are non-zero is not obvious. We prove this in Proposition~\ref{propnonzero} below.
		\end{proof}

		\begin{remark}[Relation to skein theory and the empty skein]
			If $\cat{A}$ is a ribbon fusion category, in particular semisimple, then $\widehat{\cat{A}}(H)$ can be seen as the classical handlebody skein module~\cite{roberts,masbaumroberts}, see~\cite[Section~1.7]{brochierwoike} for the connection.
		Then $\xi_H$ corresponds to the \emph{empty skein}. From this perspective, Corollary~\ref{corpointing} produces a non-semisimple analogue of the empty skein, purely from the microcosm principle.
		It should be pointed out that this construction of the non-semisimple version of the `empty skein' is really non-trivial. 
		For the moment, there is no skein-theoretic description of $\widehat{\cat{A}}(H)$  available.
		In work in progress, we will make a connection to the admissible skeins in~\cite{asm}, see~\cite{brownhaioun} for the connection to factorization homology. 
		But even then, the skein-theoretic methods would just be applied to the tensor ideal of projective objects (to which the monoidal unit does not belong in the non-semisimple case), therefore making the notion of an empty skein still somewhat tricky.
		\label{rememptyskein}
	\end{remark}
	
\begin{remark}	The invariance property of the vectors in Corollary~\ref{corpointing} under $\Map(H)$
	can generally not be extended to an invariance under the mapping class group $\Map(\partial H)$ (of which $\Map(H)$ is a subgroup): If $\cat{A}$ is modular, the handlebody group action on $\widehat{\cat{A}}(H)$ extends to a projective action of the mapping class group of the surface $\partial H$ through the construction of~\cite{lyubacmp,lyu,lyulex}, see \cite{cyclic,brochierwoike} for the relation to ansular functors.
	But 
	unless $\cat{A}\simeq \vect$,
	the consistent systems of  ansular correlators from Corollary~\ref{corpointing} will not have the 
	 invariance property for the action of the mapping class groups of $\partial H$.
	This is a consequence of \cite[Remark~4.12~(i)]{jfcs}.
	\end{remark}
	
	\begin{remark}[The correspondence between two-dimensional full conformal field theory in genus zero and three-dimensional skein theory]\label{rem2d3d}
		One of the main ideas of the works~\cite{frs1,frs2,frs3,frs4,ffrs} is to solve problems in two-dimensional rational conformal field theory through three-dimensional topological field theory. One then needs to deal with a higher-dimensional field theory that however behaves in a much simpler way.
		The attentive reader might have noticed that Theorem~\ref{thmansularcor} goes into a somewhat similar direction even though the three-dimensional topological structure that we encounter is much weaker (but in exchange applies to vastly more general situations).
		Let us make this precise: Suppose that we are given a conformal field theory in genus zero, namely
		monodromy data $\cat{A}$ (a ribbon Grothendieck-Verdier category) plus a consistent system of correlators $F$
		(a symmetric commutative Frobenius algebra in $\cat{A}$).
		 Corollary~\ref{corext} and Theorem~\ref{thmansularcor} imply that $F$ extends uniquely to a modular $\Hbdy$-algebra with coefficients in the ansular functor $\widehat{\cat{A}}$, thereby producing vectors in the spaces on conformal blocks $\widehat{\cat{A}}(H;F,\dots,F)$ associated to handlebodies $H$. But the $\widehat{\cat{A}}(H;-)$ are by \cite[Section~4.3]{brochierwoike} exactly the generalized skein modules on which the skein algebra defined in terms of factorization homology~\cite[Definition~2.6]{skeinfin} acts. 
		Phrased differently, the genus zero conformal field theory $(\cat{A},F)$ might not admit a description in terms of three-dimensional topological field theory, but there is always
		--- for purely topological reasons ---
		a description in terms of three-dimensional handlebody skein modules defined in terms of factorization homology.
		\end{remark}

	\section{A multiplicative invariant for modular algebras\label{secmult}}
	The corolla $T_1$ with two legs is sent by any $\Grpd$-valued modular operad $\cat{O}:\Graphs \to \Cat$ to the groupoid $\cat{O}(T_1)$ of unary operations. The line with two vertices on it gives us a morphism $T_1 \sqcup T_1 \to T_1$ endowing $\cat{O}(T_1)$ with the structure of a $\Grpd$-valued algebra (that can be seen as a topological algebra if needed), the \emph{algebra $\mathfrak{A}_\cat{O}$ of unary operations}~\cite[Definition~3.5]{mwansular}. 
	Its unit is the operadic identity. The $\mathbb{Z}_2$-action coming from the cyclic symmetry endows $\mathfrak{A}_\cat{O}$ with an anti involution of algebras.  We denote the multiplication in $\mathfrak{A}_\cat{O}$ by $\star$. It makes $\mathfrak{A}_\cat{O}$ into a monoidal category.
	
	\begin{theorem}\label{thmringR}
		Let $\cat{B}$ be an $\cat{S}$-valued modular $\cat{O}$-algebra. Then for any self-dual object $F$ in $\cat{B}$, 
		\begin{align}
			R: \mathfrak{A}_\cat{O} \to \cat{S}(I,I) \ , \quad o \mapsto R_o := \cat{B}_o (F,F)
			\end{align}
		is
		a lax symmetric monoidal functor. This means in particular: \begin{pnum}
			\item We have products \label{thmringRi}
		\begin{align}
			\star : R_o \otimes R_{o'} \to R_{o\star o'}
			\end{align} graded over unary operations, associative and unital relative to the associators and unitors of $\mathfrak{A}_\cat{O}$ and $\cat{S}(I,I)$, and natural in $o, o' \in \mathfrak{A}_\cat{O}$.
			\item Every $R_o$ is an $R_1$-module, where $1 \in \mathfrak{A}_\cat{O}$ is the unit, and the $\Aut(o)$-action on $R_o$ is through $R_1$-module maps.\label{thmringRii}
			\item The operations  $R_o \otimes R_{o'} \to R_{o\star o'}$ are $\Map(o) \times \Map(o')$-equivariant. \label{thmringRiii}\end{pnum}
		\end{theorem}
	
	\begin{proof}
		The functor $J: \mathfrak{A}_o \to \int \cat{O}$ sending $o \in \mathfrak{A}_o$ to $(T_1 ,    o     )$ is lax symmetric monoidal because the gluing morphism $g:T_1 \sqcup T_1 \to T_1$ gives us the structure maps
		\begin{align}
			J(o) \sqcup J(o') =(T_1,o)\sqcup (T_1,o') \ra{(    g,\id   )}      (T_1,o\star o')=J(o\star o') \ . 
			\end{align}
		The functor $R$ is just the composition
		\begin{align}
			\mathfrak{A}_\cat{O} \ra{J} \int \cat{O} \ra{\vb_F^\cat{B}}\cat{S}(I,I)   \label{eqndefR}
			\end{align} of the lax symmetric monoidal functor $J$ with the symmetric monoidal functor $\vb_F^\cat{B}$ from Proposition~\ref{propinsertionbundle}. The statements~\ref{thmringRi}-\ref{thmringRiii} just partially spell out what the lax symmetric monoidality of $R$ entails. 
		\end{proof}
		
		\begin{corollary}\label{corthmstar}
			Let in Theorem~\ref{thmringR} additionally a modular 
			$\cat{O}$-algebra structure $\xi: \star \to \vb_F^\cat{B}$
			on $F$ with coefficients in $\cat{B}$ be given.
			Then for $o,o' \in \mathfrak{A}_\cat{O}$,
			\begin{align}\xi_o \star \xi_{o'} = \xi_{o\star o'} \ . \label{eqnxistar}
				\end{align}
			\end{corollary}
	
	\begin{proof}
		The transformation $\xi \circ J$ from the lax monoidal functor $\star : \mathfrak{A}_\cat{O}\to \cat{S}(I,I)$ to $\vb_F^\cat{B} \circ J =R$ (this is how $R$ was defined, see~\eqref{eqndefR}) is monoidal because $J$ is lax symmetric monoidal and $\xi$ is by assumption monoidal. This implies~\eqref{eqnxistar}.
		\end{proof}
	
	The following gives us a far-reaching generalization of the multiplicative structures built in
	\cite[Section~5.2]{juhasz}
	on the state spaces
	of a three-dimensional topological field theory:
	
	\begin{corollary}\label{corproducts}
		Let $\cat{B}$ be a $\Lexf$-valued
		ansular functor such that the ribbon Grothendieck-Verdier category obtained by evaluation on the circle comes with a self-dual unit. 
		Then there are natural unital and associative product maps
		\begin{align}
			\star : \cat{B}(H) \otimes \cat{B}(H') \to \cat{B}(H \# H') \label{eqnHHstar}
			\end{align} for handlebodies without embedded disks, where $\#$ is the connected sum. 
		\end{corollary}
	The spaces of conformal blocks appearing here are for handlebodies without embedded disks in their boundary, but in order to have the operation $\star$, one still needs to fix disks at which the connected sum is taken.
	The equivariance of the $\star$-operation 
	in the sense of Theorem~\ref{thmringR}~\ref{thmringRiii}
	 is  for the handlebodies groups of the handlebodies with these disks.

	\begin{proof}[\slshape Proof of Corollary~\ref{corproducts}]
		This follows if we apply Theorem~\ref{thmringR} to the self-dual
		object $I$ 
		and by noting that  $\cat{B}(H)\cong \cat{B}(H_2;I,I)$ for every handlebody $H$ without embedded disks, where $H_2$ is the same handlebody with two disks embedded in its boundary. 
		 	\end{proof}
	
		For the direct sum $\bigoplus_{g\ge 0} \widehat{\cat{A}}(H_g)$, this implies the following consequence:
		
	\begin{corollary}\label{corproduct}
		For any ribbon Grothendieck-Verdier category $\cat{A}$ in $\Lexf$ with self-dual monoidal unit,
		the direct sum of the spaces of conformal blocks at all genera (without boundary components) is naturally an algebra graded over the genus.
		By equipping the spaces of conformal blocks for all handlebodies $H$ with this multiplicative structure, we obtain a strictly finer invariant of the ansular functor than just the representations on all handlebodies $H$ without embedded disks.
	\end{corollary}

\begin{proof} The first part follows directly from Corollary~\ref{corproducts}. The second part will be a consequence of Example~\ref{exmult}.\end{proof}

\begin{example}\label{exmult}
	We thank Ingo Runkel and Christoph Schweigert for bringing this example to our attention. 
	Let $F$ be a finite-dimensional commutative algebra and $\cat{A}$ the category of finite-dimen\-sional $F$-modules equipped with the symmetric monoidal product $\odot$
	defined by $X \odot Y := (Y^* \otimes_F X^*) ^*$ for $F$-modules $X$ and $Y$, where $-^* : \cat{A}\to\cat{A}^\op$ takes the contragredient $F$-module, namely the linear dual, seen as right $F$-module and then also left $F$-module by commutativity.
	Now $(\cat{A},\odot,*)$ is a Grothendieck-Verdier category in $\Lexf$
	(with $\otimes_F$, $\cat{A}$ would be a Grothendieck-Verdier category in $\Rexf$; the 
	second monoidal product $\odot$
	obtained by conjugation with the duality as in~\cite{bd}, see also~\eqref{eqn2ndmonoidalproduct} above, is then left exact).
	The product $\odot$ is symmetric. With the trivial braiding and the trivial balancing, $\cat{A}$ becomes a ribbon Grothendieck-Verdier category in $\Lexf$. 
	By seeing $F$ as a vertex operator algebra, this can be viewed as a special case of the construction of the ribbon Grothendieck-Verdier structure on categories of modules of vertex operator algebras~\cite[Theorem~2.12]{alsw}. With the Peter-Weyl Theorem in the form of~\cite[Corollary~2.9]{fss}, we find $\int^{X \in \cat{A}} X^* \boxtimes X \cong F^*$ as objects in $\cat{A}\boxtimes \cat{A}$, i.e.\ as $F$-bimodules. 
	Let us now assume that $F$ is even a Frobenius algebra, so that $F^* \cong F$ as $F$-bimodules. 
	Then we find $\mathbb{F} = \odot \left(  \int^{X \in \cat{A}} X^* \boxtimes X \right)\cong F$. In other words, the canonical coend is the monoidal unit $F \in \cat{A}$. This entails that the space of conformal blocks for the handlebody $H_g$ of genus $g$ without embedded disks is given by $\widehat{\cat{A}}(H_g)\cong \cat{A}(I,\mathbb{F}^{\otimes g})=\Hom_F(F,F)\cong F$ thanks to \cite[Theorem~7.8]{cyclic}; the handlebody group representation is trivial.
	In fact, $\cat{A}$ produces not only an ansular, but a modular functor, even though $\cat{A}$ is generally not rigid.
	The spaces of conformal blocks $\widehat{\cat{A}}(H_g)$ are the same as for the ribbon Grothendieck-Verdier category $\vect^{\oplus n}$ with $n:=\dim \, F$, but the multiplicative structure on $\bigoplus_{g\ge 0} \widehat{\cat{A}}(H_g)$, which is induced by the multiplication of $F$, is different from the corresponding one for $\vect^{\oplus n}$ if $F$ is not semisimple. 
	Note also that if $F$ is not semisimple, $\cat{A}$ and $\vect^{\oplus n}$ give us an easy example of two ribbon Grothendieck-Verdier categories that
	produce non-equivalent modular functors whose mapping class group representations of closed surfaces, however, agree at all genera. This is, in a rather generalized sense, another instance of the fact that `modular data' is generally insufficient for the description of modular functors~\cite{mignardschauenburg}.
	\end{example}

	\begin{proposition}\label{propnonzero0}
		Let $\cat{A}$ be a finite ribbon category.
		Then the products
		\begin{align}
		\star : \widehat{\cat{A}}(H) \otimes \widehat{\cat{A}}(H') \to \widehat{\cat{A}}(H\# H')    \label{eqntheproducts}
		\end{align} are injective.
		\end{proposition}
	
Without the hypotheses that $\cat{A}$ is finite ribbon, this is generally wrong:	In the Example~\ref{exmult}, the products~\eqref{eqntheproducts} are given by
 the multiplication operation $F \otimes F \to F$ of the commutative Frobenius algebra $F$. Clearly, this map is  not injective if $\dim\, F \ge 2$.

	\begin{proof}[\slshape Proof of Proposition~\ref{propnonzero0}]We use \cite[Theorem~7.8]{cyclic} to see that,
		after choosing a cut system for $H$ with genus $g$ and $H'$ with genus $g'$, the map $\star$ is given by the map
		\begin{equation}
			\begin{tikzcd}
				\cat{A}(I,\mathbb{F}^{\otimes g}) \otimes \cat{A}(I,\mathbb{F}^{\otimes g'})   \ar[]{rrr}{\cong} \ar[swap]{dd}{\star } & &&         \cat{A}(I,\mathbb{F}^{\otimes g}\otimes I) \otimes \cat{A}(I,I\otimes \mathbb{F}^{\otimes g'})
				\ar{dd}{  I\boxtimes I \to \Delta=\Delta'\boxtimes \Delta''=\int^{X \in \cat{A}}X^\vee \boxtimes X  }    \\
				\\ \cat{A}(I,\mathbb{F}^{\otimes (g+g')}) 
				& &  & \ar{dd}{\text{$\cong$, see \cite[Proposition~2.28]{cyclic}}}
		  	\cat{A}(I,\mathbb{F}^{\otimes g}\otimes \Delta') \otimes \cat{A}(I,\Delta''\otimes \mathbb{F}^{\otimes g'}) \\ \\ &&& \ar{uulll}{\text{composition over $X$}}	\oint^{X \in \cat{A}}\cat{A}(X,\mathbb{F}^{\otimes g} ) \otimes \cat{A}(I,X \otimes \mathbb{F}^{\otimes g'})  \ . 
			\end{tikzcd} 
		\end{equation}
	Here $\oint$ is the left exact coend~\cite{lyulex}, see~\cite[Theorem~6.4]{cyclic} for the connection 
	to the gluing of modular algebras.
		From this, we deduce that the multiplication map $\star$
		\begin{align}
		\cat{A}(I,\mathbb{F}^{\otimes g}) \otimes \cat{A}(I,\mathbb{F}^{\otimes g'})\to \cat{A}(I,\mathbb{F}^{\otimes (g+g')}) \label{eqnmultiplicativestrcture}
		\end{align} is just the map induced by $\otimes$.
		Since $I$ is simple, this map is injective by \cite[Lemma~4.6]{mwdehn}.
		\end{proof}

Note that the multiplicative structure can, after choosing a cut system, 
be written down directly via~\eqref{eqnmultiplicativestrcture}. 
With this purely algebraic description, the independence of the cut system and the equivariance with respect to the mapping class group actions that is a part of Corollary~\ref{corproducts} seem hard to see.
	The question whether there are products on spaces of conformal blocks with respect to the connected sum ---
	a structure that is not part of the axioms of ansular or modular functors --- was posed to us by Simon Lentner and Christoph Schweigert. 
	The results of this section, in particular Corollary~\ref{corproducts}, provide a purely topological construction of these operations.

	\begin{proposition}\label{propnonzero}
		For a finite ribbon category $\cat{A}$, the distinguished $\Map(H)$-invariant vectors $\xi_H \in \widehat{\cat{A}}(H)$ are non-zero for all handlebodies and satisfy \begin{align} \xi_H \star \xi_{H'} = \xi_{H \# H'}\ . \label{eqnconnectedsumxiH}
			\end{align}
		\end{proposition}
		
	One might first think that $\cat{A}=0$ has to be excluded, but the trivial linear monoidal category is not a finite ribbon category because $\End_\cat{A}(I)\neq k$ in that case.
	
	\begin{proof}
		First we observe that Corollary~\ref{corthmstar} gives us~\eqref{eqnconnectedsumxiH}. 
		
		Let us now prove $\xi_H \neq 0$:
		By construction $\xi_{\mathbb{B}^3} \in \cat{A}(I,I)$ is $\id_I \neq 0$ and $\xi_{\mathbb{D}^2 \times \mathbb{S}^1} \in \cat{A}(I,\mathbb{F})$ is the map $I\cong I^\vee \otimes I \to \int^{X\in \cat{A}}X^\vee \otimes X$ which is unit of Lyubashenko's algebra structure on $\mathbb{F}$~\cite[Section~2]{lyu} and hence non-zero.
		If $H$ has genus $g$, then $\xi_H=\xi_{\mathbb{D}^2 \times \mathbb{S}^1}^{\star g}$ by~\eqref{eqnconnectedsumxiH}. This is non-zero by Proposition~\ref{propnonzero0} since we have just seen $\xi_{\mathbb{D}^2 \times \mathbb{S}^1}\neq 0$.
		\end{proof}

		\begin{lemma}\label{lemmadim}
			Let $\cat{A}$ be a unimodular finite ribbon category. If for any handlebody $H$ with genus $g$
			\begin{align}
				\dim\, \widehat{\cat{A}}(H)< 2^g 
				\end{align}
				for the space of conformal blocks,
				then $\cat{A}\simeq \vect$.
			\end{lemma}

		\begin{proof}
			We show $\dim\, \widehat{\cat{A}}(H)\ge 2^g $ if $\cat{A} \neq \vect$.
			In fact, thanks to $\dim\, \widehat{\cat{A}}(H) \cdot \dim\, \widehat{\cat{A}}(H) \le \dim\, \widehat{\cat{A}}(H \# H')$ (see Proposition~\ref{propnonzero0}), it suffices to prove $\dim\, \widehat{\cat{A}}(\mathbb{D}^2 \times \mathbb{S}^1)\ge 2$. By \cite[Theorem~7.8]{cyclic} $\widehat{\cat{A}}(\mathbb{D}^2 \times \mathbb{S}^1)\cong \cat{A}(I,\mathbb{F})$, which is at least two-dimensional by~\cite[Lemma~5.7]{shimizumodular} if $\cat{A}\neq \vect$.
			\end{proof}
		
	\begin{theorem}\label{thmnotrrep}
		Let $\cat{A}$ be a unimodular finite ribbon category.
		If $\cat{A}\neq \vect$ and $H$ is a handlebody without embedded disks and genus $g\ge 1$,
		then the $\Map(H)$-representation  $\widehat{\cat{A}}(H)$ is not irreducible.
		\end{theorem}

	Clearly, if $\cat{A}=\vect$ or $H=\mathbb{B}^3$, the representation would be $k$ with trivial action and hence  irreducible.
	
	\begin{proof}[\slshape Proof of Theorem~\ref{thmnotrrep}]
		The $\Map(H)$-representation  $\widehat{\cat{A}}(H)$ contains the non-zero subrepresentation $k \cdot \xi_H$ (with trivial $\Map(H)$-action)
		 by Proposition~\ref{propnonzero}. Thanks to $\cat{A}\neq \vect$, this is a proper subrepresentation by Lemma~\ref{lemmadim}.
		\end{proof}

	\begin{example}[Symmetric categories]
		Let $\cat{A}$ be a finite ribbon category.
		By Proposition~\ref{propnonzero0} we have an injective map
		\begin{align}
		\star :	\widehat{\cat{A}}(H_1) ^{\otimes g} \to \widehat{\cat{A}}(H_g) \ .
			\end{align}
		Let us think of the handlebody $H_g$ as embedded in $\mathbb{R}^3$.
		 The map $\star$ becomes $\mathbb{Z}_g$-equivariant if we let the generator $1 \in \mathbb{Z}_g$ 
		 act on $\widehat{\cat{A}}(H_1) ^{\otimes g}$ by
		$v_1 \otimes v_2 \otimes \dots \otimes v_g \mapsto v_g \otimes v_1 \otimes \dots \otimes v_{g-1}$ (cyclic permutation),
		 and act on $\widehat{\cat{A}}(H_g)$ by a rotation of $H_g$ by $2\pi /g$
		 as indicated
		 in the following picture (see also~\cite[Section~13.2.1]{farbmargalit}):
		 	\begin{equation}
		 	\begin{tikzpicture}[scale=0.5]
		 		\begin{pgfonlayer}{nodelayer}
		 			\node [style=none] (0) at (-10, 3) {};
		 			\node [style=none] (1) at (-7, 3) {};
		 			\node [style=none] (2) at (-9, 3) {};
		 			\node [style=none] (3) at (-8, 3) {};
		 			\node [style=none] (4) at (-7.75, 3.25) {};
		 			\node [style=none] (5) at (-9.25, 3.25) {};
		 			\node [style=none] (6) at (-8, 0.5) {};
		 			\node [style=none] (7) at (-5, 0.5) {};
		 			\node [style=none] (8) at (-7, 0.5) {};
		 			\node [style=none] (9) at (-6, 0.5) {};
		 			\node [style=none] (10) at (-5.75, 0.75) {};
		 			\node [style=none] (11) at (-7.25, 0.75) {};
		 			\node [style=none] (12) at (-12, 0.5) {};
		 			\node [style=none] (13) at (-9, 0.5) {};
		 			\node [style=none] (14) at (-11, 0.5) {};
		 			\node [style=none] (15) at (-10, 0.5) {};
		 			\node [style=none] (16) at (-9.75, 0.75) {};
		 			\node [style=none] (17) at (-11.25, 0.75) {};
		 			\node [style=none] (18) at (5.5, 3.25) {};
		 			\node [style=none] (19) at (6.5, 3.25) {};
		 			\node [style=none] (20) at (6.75, 3.5) {};
		 			\node [style=none] (21) at (5.25, 3.5) {};
		 			\node [style=none] (22) at (4, 0.75) {};
		 			\node [style=none] (23) at (5, 0.75) {};
		 			\node [style=none] (24) at (5.25, 1) {};
		 			\node [style=none] (25) at (3.75, 1) {};
		 			\node [style=none] (26) at (7, 0.75) {};
		 			\node [style=none] (27) at (8, 0.75) {};
		 			\node [style=none] (28) at (8.25, 1) {};
		 			\node [style=none] (29) at (6.75, 1) {};
		 			\node [style=none] (30) at (6, -1.25) {};
		 			\node [style=none] (32) at (6, 4.5) {};
		 			\node [style=none] (33) at (-2, 1.75) {};
		 			\node [style=none] (34) at (2, 1.75) {};
		 			\node [style=none] (35) at (0, 2.25) {$\star$};
		 			\node [style=none] (36) at (-4, 0.5) {};
		 			\node [style=none] (37) at (-5.25, 3.75) {};
		 			\node [style=none] (38) at (10, -1) {};
		 			\node [style=none] (39) at (10, 4) {};
		 			\node [style=none] (40) at (12.75, 2) {rotation};
		 			\node [style=none] (41) at (-2, 3.5) {permutation};
		 		\end{pgfonlayer}
		 		\begin{pgfonlayer}{edgelayer}
		 			\draw [bend left=90, looseness=1.25] (0.center) to (1.center);
		 			\draw [bend right=90, looseness=1.25] (0.center) to (1.center);
		 			\draw [bend left=90] (2.center) to (3.center);
		 			\draw [bend right=75, looseness=1.25] (5.center) to (4.center);
		 			\draw [bend left=90, looseness=1.25] (6.center) to (7.center);
		 			\draw [bend right=90, looseness=1.25] (6.center) to (7.center);
		 			\draw [bend left=90] (8.center) to (9.center);
		 			\draw [bend right=75, looseness=1.25] (11.center) to (10.center);
		 			\draw [bend left=90, looseness=1.25] (12.center) to (13.center);
		 			\draw [bend right=90, looseness=1.25] (12.center) to (13.center);
		 			\draw [bend left=90] (14.center) to (15.center);
		 			\draw [bend right=75, looseness=1.25] (17.center) to (16.center);
		 			\draw [bend left=90] (18.center) to (19.center);
		 			\draw [bend right=75, looseness=1.25] (21.center) to (20.center);
		 			\draw [bend left=90] (22.center) to (23.center);
		 			\draw [bend right=75, looseness=1.25] (25.center) to (24.center);
		 			\draw [bend left=90] (26.center) to (27.center);
		 			\draw [bend right=75, looseness=1.25] (29.center) to (28.center);
		 			\draw [bend left=90, looseness=1.75] (30.center) to (32.center);
		 			\draw [bend left=90, looseness=1.75] (32.center) to (30.center);
		 			\draw [style=end arrow] (33.center) to (34.center);
		 			\draw [style=end arrow, bend right=45] (36.center) to (37.center);
		 			\draw [style=end arrow, bend right=45] (38.center) to (39.center);
		 		\end{pgfonlayer}
		 	\end{tikzpicture}
		 \end{equation}
		  Here the numbering of the copies of $H_1$ and the sense of the rotation have to be chosen accordingly of course. 
		If $\cat{A}$ is not $\vect$, then the $\mathbb{Z}_g$-action on $\widehat{\cat{A}}(H_1) ^{\otimes g}$ is faithful thanks to
		Lemma~\ref{lemmadim}. By injectivity of $\star$ the same is true for the rotation action on  $ \widehat{\cat{A}}(H_g)$.

		Suppose now that our finite ribbon category $\cat{A}$ is symmetric, and we choose the identity as ribbon structure. By what we have just seen the handlebody group action on $\widehat{\cat{A}}(H_g)$ cannot be trivial unless $\cat{A}=\vect$. Clearly, all elements in the twist group, see \eqref{eqnluft}, act trivially, but the action of $\catf{Out}(F_g)$ can still be non-trivial! 
		If the action on $\widehat{\cat{A}}(H_g)$ extended to the mapping class group of $\partial H_g$, then this could only be a trivial action because $\Map(\partial H_g)$ is generated by Dehn twists, and all of these act trivially \cite[Theorem 4.8]{mwdehn} --- a contradiction.
		We conclude that a finite symmetric tensor category that is not $\vect$ does not extend to a modular functor. 
		This might be a little bit surprising because a first na\" ive intuition might tell us that we can just build a modular functor with trivial mapping class group actions. 
		Note that by Example~\ref{exmult} there are actually symmetric ribbon Grothendieck-Verdier categories that extend to a modular functor, but  \emph{rigid} ones with simple unit do not, unless they are $\vect$!
		\end{example}

	\section{A moduli space for open conformal field theories}
	In this last section, we explain
	how the modular microcosm principle can be used to construct and study 
	\emph{moduli spaces of conformal field theories}. This will be worked out here for \emph{open} conformal field theories; extensions to other situations will be considered elsewhere.
	An idea developed most notably in~\cite{frs1,frs2,frs3,frs4,ffrs}
	is that conformal field theories are described as a pair of a category encoding the monodromy data (this can be for example a tensor category coming from a vertex operator algebra) and a consistent system of correlators (that takes the form of some sort of Frobenius algebra). 
	These pairs can be organized in a moduli space as follows (we consider the open case):
	By sending an open modular functor $\mathfrak{B}$ to the modular algebras with coefficients in $\mathfrak{B}$,
	we obtain a functor $\catf{ModAlg}(\O;-):\catf{ModAlg}(\O) \to \Cat$. 	
	This allows us to define a \emph{moduli space of open conformal field theories (in $\Lexf$)} by taking the nerve and realization of the Grothendieck construction of this functor
	\begin{align}\label{eqndefopencft}
	\catf{OpenCFT} := \left|B\int \left(  \catf{ModAlg}(\O;-):\catf{ModAlg}(\O) \to \Cat  \right)\right| \ . 
	\end{align}
	By its definition and Theorem~\ref{thmopencorrelators} $\catf{OpenCFT}$ is a space whose points are pairs 
	$C=(\cat{A},F)$
	of the monodromy data of the open conformal field theory (which is a pivotal Grothendieck-Verdier category $\cat{A}$)
	and a consistent system of open correlators (which is a symmetric Frobenius algebra $F$ in $\cat{A}$).
	We denote by $\catf{OpenCFT}|_{\cat{A}} \subset \catf{OpenCFT}$ the subspace of
	 $\catf{OpenCFT}$ consisting of those open conformal field theories 
	whose monodromy data is (equivalent to) $\cat{A}$.
	
\begin{theorem}\label{thmseven} Let $C=(\cat{A},F)$ be an open conformal field theory given by its monodromy data, namely a pivotal Grothendieck-Verdier category $\cat{A}$ in $\Lexf$ and a symmetric Frobenius algebra $F$ in $\cat{A}$. Then there is the following seven-term exact sequence
	\begin{center}
		\begin{equation}\centering
			\hspace*{-3cm}\begin{tikzcd}
				1\ar[r] &
				\pi_2(\catf{OpenCFT},C) 	\ar[r]
				& \text{\footnotesize cycl. autom. of $\id_\cat{A}$}    \ar[out=355,in=175,dll]    \\
				\catf{Aut}_\text{sym Frob. alg.}(F) 	 \ar[r] & \pi_1(\catf{OpenCFT},C) \ar[r] & \text{\footnotesize cycl. autom. of $\cat{A}$}/\! \cong \ar[out=355,in=175,dll] \\ \text{\footnotesize sym. Frob. alg. in $\cat{A}$}/\! \cong   \ar[r] &   \text{\footnotesize open CFTs w. monodr.  $\cat{A}$}/\! \cong \ar[r] & \star
			\end{tikzcd} 
	\end{equation}\end{center}
	in which the last two non-trivial terms are pointed sets.
\end{theorem}

We will prove this result by realizing that this seven-term
exact sequence is the long exact sequence of homotopy groups for the following fiber sequence:

\begin{proposition}\label{propfiberseq}
	For any pivotal Grothendieck-Verdier category $\cat{A}$ in $\Lexf$, there is a homotopy fiber sequence
	\begin{align}
	|B\catf{SymFrobAlg}(\cat{A})|\to	\catf{OpenCFT}|_{\cat{A}} \to |B \catf{AUT}(\cat{A})|
		\end{align}
	\end{proposition}

\begin{proof}
		Thanks to \cite[Theorem~1.2]{thomason}, there is a homotopy equivalence
	\begin{align}
	\catf{OpenCFT} \simeq \hocolimsub{\mathfrak{B} \in \catf{ModAlg}(\O) } |B\catf{ModAlg}(\O;\mathfrak{B})| \ . 
	\end{align}
	From Theorem~\ref{thmopencorrelators}, we can now conclude that $\catf{OpenCFT}|_{\cat{A}}$ is homotopy equivalent to the homotopy quotient
	$|B\catf{SymFrobAlg}(\cat{A})| // \catf{AUT}(\cat{A})$ of the 
	the space $|B\catf{SymFrobAlg}(\cat{A})|$ by the 2-group $\catf{AUT}(\cat{A})$.
	This gives us the desired homotopy fiber sequence.
	\end{proof}

\begin{example}
	Let $\cat{A}$ be a pivotal finite tensor category. Consider the symmetric Frobenius algebra $I\in \cat{A}$, and the open conformal field theory $C=(\cat{A},I)$.
	Then $\catf{Aut}_\text{sym Frob. alg.}(I)=1$, which implies that $\pi_2(\catf{OpenCFT},C)$ is the group of cyclic automorphisms of $\id_\cat{A}$, i.e.\ monoidal automorphisms $\lambda_X:X\to X$ with $\lambda_X^\vee = \lambda_{X^\vee}$ for all $X\in \cat{A}$. Since all cyclic automorphisms of $\cat{A}$ preserve the symmetric Frobenius algebra $I$, $\pi_1(\catf{OpenCFT},C)$ is the group of cyclic automorphisms of $\cat{A}$ up to isomorphism. 
	\end{example}

\small
\newcommand{\etalchar}[1]{$^{#1}$}

	\vspace*{0.5cm}
	\noindent \textsc{Université Bourgogne Europe, CNRS, IMB UMR 5584, F-21000 Dijon, France}

\end{document}